\documentclass[12pt,twoside,twocolumn,balance]{article}

\usepackage{diagbox} 
\usepackage[utf8]{inputenc}
\usepackage{xcolor}
\usepackage{amsfonts}
\usepackage{amsmath}
\usepackage{amssymb}
\usepackage{hyperref}
\usepackage[colors,economics]{optsys}
\usepackage{optsys_hympulsion}
\usepackage[labelfont=bf]{caption}
\usepackage{tabularx}
\usepackage{algorithm}
\usepackage{longtable}
\usepackage{geometry}
\usepackage{algpseudocode}
\usepackage[autolanguage]{numprint} 
\usepackage{siunitx}
\usepackage{empheq}
\usepackage{cases}
\usepackage{cleveref}
\usepackage{enumitem} 
\hypersetup{
  colorlinks=true, 
  linktoc=all,     
  linkcolor=blue,
  urlcolor=blue,
  citecolor=blue
}
\usepackage{cite}
\usepackage{wasysym} 
\usepackage{tikz}
\usetikzlibrary{fit}
\usepackage{float}
\usetikzlibrary{automata,arrows,positioning,calc,decorations}

\newcommand{\HOUR}{\mathbb{H}}

\newcommand{\hour}{h}

\newcommand{\cday}{j}
\newcommand{\Lday}[1]{\underline{#1}}
\newcommand{\Rday}[1]{\overline{#1}}
\newcommand{\lday}{\Lday{\cday}}
\newcommand{\rday}{\Rday{\cday}}
\newcommand{\LDAYS}{\Lday{\DAYS}}
\newcommand{\RDAYS}{\Rday{\DAYS}}
\newcommand{\EDAYS}{\widehat{\DAYS}}
\newcommand{\LARCS}{\underline{\ARCS}}
\newcommand{\RARCS}{\overline{\ARCS}}
\newcommand{\SARCS}{\widetilde{\ARCS}}

\newcommand{\TEG}{\textsf{te}} 

\newcommand{\algname}[1]{\texttt{\textbf{#1}}}

\newcommand{\kfun}{\kappa}
\newcommand{\StockMax}{\overline{\Stock}}
\newcommand{\StockCritic}{\Stock^{\textsc{cr}}}
\newcommand{\MaxArcsSource}{N_{\source}}

\newcommand{\bijection}[2]{\sigma_{#2,#1}}
\newcommand{\obijection}[2]{\overline{\sigma}_{#2,#1}}
\newcommand{\tbijection}[2]{\widetilde{\sigma}_{#2,#1}}
\newcommand{\hbijection}[2]{\widehat{\sigma}_{#2,#1}}
\newcommand{\TransportPlanning}{T}
\newcommand{\IntegerSet}[1]{[#1]}
\newcommand{\ssarc}[1]{\source{:}k}
\graphicspath{{./}}
\usepackage{a4wide}
\usepackage{fullpage}
\usepackage{balance}
\ifdefined\result\else \newtheorem{result}{Result}\fi
\newcommand{\optional}[1]{}
\newcommand{\rl}[1]{#1}

\def\keywords#1{\par\addvspace\medskipamount{\rightskip=0pt plus1cm
\def\and{\ifhmode\unskip\nobreak\fi\ $\cdot$
}\noindent\keywordname\enspace\ignorespaces#1\par}}
\def\keywordname{{\bfseries Keywords}}%

\title{A Production Routing Problem with Mobile Inventories}

\author{Raian Lefgoum\thanks{Cermics, École Nationale des Ponts et Chaussées, IP Paris, 6 et 8 avenue Blaise Pascal, 77455 Marne la Vallée Cedex 2}
  \and
  Sezin Afsar\thanks{Universidad de Oviedo, Gijón, Principality of Asturias, Spain}
  \and Pierre Carpentier\thanks{UMA, ENSTA Paris, IP Paris, France}
  \and Jean-Philippe Chancelier$^{*}$
  \and Michel De Lara$^{*}$}

\begin{document}
\maketitle


\begin{abstract}
  Hydrogen is an energy vector, and one possible way to reduce CO$_{2}$
  emissions. This paper focuses on a hydrogen transport problem where mobile
  storage units are moved by trucks between sources to be refilled and
  destinations to meet demands, involving swap operations upon arrival. This
  contrasts with existing literature where inventories remain stationary.
  The objective is to optimize daily routing and refilling schedules of the
  mobile storages. We model the problem as a flow problem on a time-expanded
  graph, where each node of the graph is indexed by a time-interval and a
  location and then, we give an equivalent Mixed Integer Linear Programming
  (MILP) formulation of the problem.
  For small to medium-sized instances, this formulation can be efficiently solved using standard MILP solvers. However, for larger instances, the computational complexity increases significantly due to the highly combinatorial nature of the refilling process at the sources. To address this challenge, we propose a two-step heuristic that enhances solution quality while providing good lower bounds.
  Computational experiments demonstrate the effectiveness of the proposed
  approach. When  executed over the large-sized instances of the problem, the median (resp. the
  mean) obtained by this two-step heuristic outperforms the median (resp. the mean)
  obtained by the direct use of the MILP solver by \numprint[\%]{12.7}
  (resp. \numprint[\%]{9}), while also producing tighter lower bounds. These
  results validate the efficiency of the method for optimizing hydrogen
  transport systems with mobile storages and swap operations.

\end{abstract}


\section{Introduction}

The global shift towards sustainable energy solutions has underscored the
importance of hydrogen as a key component in reducing CO\textsubscript{2}
emissions and addressing climate change concerns. Hydrogen not only offers a
clean alternative fuel source, but also enhances energy security and
flexibility. To support the expanding hydrogen economy, efficient transportation
and distribution systems are essential for meeting (mobility) demands while
minimizing operational costs.

 In this paper, we address an optimal transport problem in which trucks are used to
transport mobile hydrogen storage units (tanks) between \emph{sources} and
\emph{destinations}. The sources represent supply locations where the storages
are refilled with hydrogen, while the destinations correspond to sites where the
stored hydrogen is consumed to meet mobility demands. The transport process
operates bidirectionally, as storages are transported from sources to
destinations to satisfy the demand and storages are also transported
back from destinations to sources for refilling, ensuring a continuous supply
cycle. Within a specified planning horizon, the transport problem objective is to find daily routing and refilling schedules of mobile storage units, in order
to minimize an intertemporal cost taking into account transport, hydrogen
purchase and demand dissatisfaction costs.

This work focuses on a variant of the Production Routing Problem (PRP), where
mobile storages are moved by trucks between sources and destinations. Unlike
conventional routing problems that assume static
inventories~\cite{IRP_transshipment,IRP_pickup_delivery} at delivery locations,
we use the concept of mobile inventories, allowing for dynamic relocation and
swap operations upon delivery which are specific to hydrogen transport. This
added layer of complexity requires optimizing daily routing and refilling
schedules for mobile storage units. 

Another distinguishing feature of this work, compared to the PRP found in the literature, is the highly combinatorial nature of the refilling process at the sources. This complexity arises from three factors: the constraint that exactly one storage unit must be present at each destination at any given time to satisfy demand, the constraint on the refilling capacity at the sources, and the fact that the stock of mobile storage units at the sources cannot be mixed. Specifically, when a new storage unit arrives from a source to replace the one previously stationed at the destination, the departing storage may not be empty, as it must be replaced immediately upon the arrival of the new mobile storage. As a result, multiple partially filled storage units may return to the same source, these units may also depart the next day without reaching full capacity because of refilling capacity at the sources. This, combined with the restriction on stock mixing, increases the complexity of the refilling process, as discussed later in~\S\ref{constraints_node_source_1}.

Now, we describe some assumptions that we make on the movements of storage units.
When a storage unit is transported by truck from a source location to a
destination location, an operation we refer to as a \emph{swap upon arrival} is
to be performed at destination: the truck exchanges the storage unit it
carries with a storage unit located at the destination. Then, the newly
positioned storage unit serves to meet demand at destination, while the
collected storage returns to a source for refilling, transported by the
same truck. To simplify the problem, we assume that \emph{the trip initialized
  by a truck on a specific day is completed within the same day}.  Moreover,
we assume that, at initial time, there is a single
storage unit at each destination. Then, during the optimization time span, we
always have \emph{a single storage unit to meet demand at each destination} (see
Fig.~\ref{fig:two_parts}).
\begin{figure*}
  \begin{center}
    \includegraphics[width=0.4\textwidth]{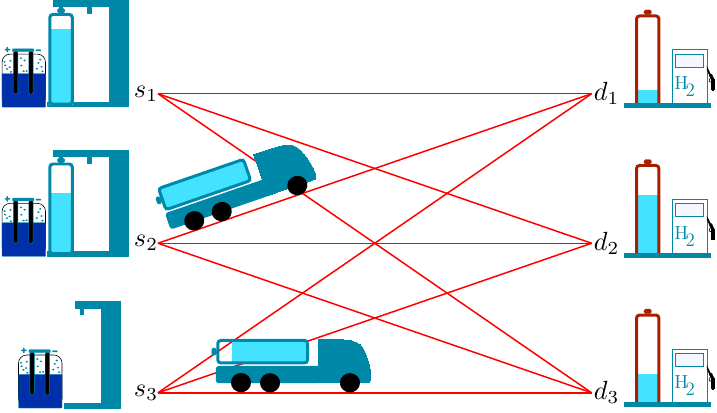}%
    \hspace{0.1\textwidth}%
    \includegraphics[width=0.4\textwidth]{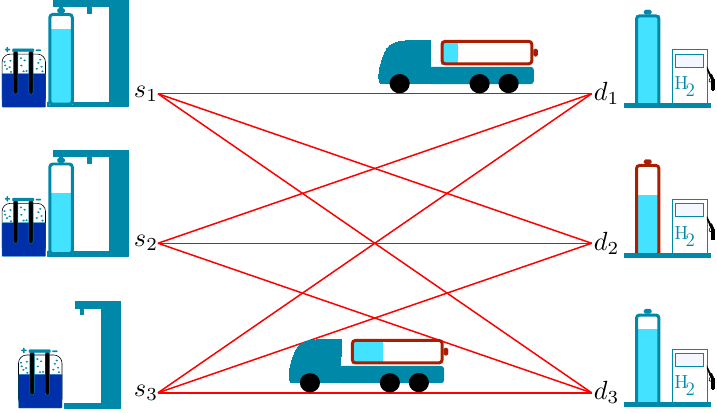}
  \end{center}
  \caption{Storage unit delivery (left) and back to source after swap (right)}
  \label{fig:two_parts}
\end{figure*}
Moreover, we assume that each destination can receive at most one storage per
day, that is, there is at most one swap per
destination and per day. Using the previous assumptions, the delivery planning
of storage units is done on a daily basis. Firstly, we need to determine what
happens to the storage units at the source locations: whether they stay
where they are or they be delivered to a destination. Second, we decide
which source location the collected storage unit goes to.

Daily refill planning of the storage units involves determining the quantity of
hydrogen to be purchased per source location per day. The daily maximum quantity
of hydrogen purchased is assumed to be bounded. Moreover, when several storage
units are present at a given source, the allocation of the purchased hydrogen
between the present storage units is to be decided.

The rest of the paper is organized as follows. We review the existing literature
in this field in Sect.~\ref{se:litterature_review}. Afterwards, in
Sect.~\ref{se:notations}, we give the notations and assumptions for our
problem. We introduce the problem on a time-expanded graph in
Sect.~\ref{se:problem_description_expanded_graph}, and give its model in
Sect.~\ref{sec:modeling}. In Sect.~\ref{se:formulation_solving}, we give the
problem formulation and describe in detail the different algorithms.
Numerical experiments are presented in Sect.~\ref{se:results}. In Sect.~\ref{sec:flow_to_planning}, we demonstrate how the storage routes (transport planning) can be derived from a solution of the problem on this time-expanded flow graph and finally the
conclusion and possible directions for future works are discussed in
Sect.~\ref{se:conclusion}.

\section{Literature review}\label{se:litterature_review}
Logistics and transport optimization have long been critical areas of research, with the \emph{Vehicle Routing Problem} (VRP) (see for example~\cite{VRP2,VRP3,VRP1}) serving as a foundational framework for optimizing the delivery of goods and services. Extending the VRP, the \emph{Inventory Routing Problem} (IRP)~\cite{thirty_years_IRP} integrates inventory management with routing decisions, while the \emph{Production Routing Problem} (PRP)~\cite{PRP_review} further incorporates production planning.

As a generalization of the VRP, the IRP and PRP share
its classification as an NP-hard problem~\cite{VRP_NP_HARD} which means that
there is no known algorithm that can solve the IRP and PRP  efficiently (in polynomial
time).

Moreover, aside from the classical IRP and PRP, a wide range of variants have
been studied in the literature. As an example, the IRP with
transshipment~\cite{IRP_transshipment} and PRP with transshipment~\cite{PRP_transshipment} enable the transfer of goods between delivery locations, offering greater flexibility but introducing
additional decision variables, whereas the IRP with pickup and delivery~\cite{IRP_pickup_delivery} and PRP with pickup and delivery~\cite{PRP_pickup_delivery} allow vehicles to visit pickup locations during their routes. These pickups can either be used for future
deliveries or stored at the depot after completing the route. In~\cite{real_life_IRP}, the authors propose a randomized local
search algorithm to address large-scale instances of the IRP, that incorporates
practical constraints such as time windows, driver safety, and site
accessibility for vehicle resources. Additionally, they introduce a surrogate
objective that ensures that short-term optimization solutions yield favorable
long-term outcomes.

Several studies have examined real-life routing problems within the gas and
hydrogen sectors
\cite{IRP_GAS3,IRP_GAS2,IRP_GAS1,ADMM,PRP_GAS1}. In~\cite{IRP_GAS3}, a model for
the liquefied natural gas sector is explored, introducing a decomposition
algorithm using a path flow model with pre-generated duties and added valid
inequalities. In~\cite{IRP_GAS2}, the Stochastic Cyclic Inventory Routing
Problem is addressed with supply uncertainty, applied to green hydrogen
distribution in the Northern Netherlands. The study proposes a combined approach
using a parameterized Mixed Integer Programming model for static, periodic
vehicle transportation schedules and a Markov Decision Process to optimize
dynamic purchase decisions under stochastic supply and
demand. In~\cite{PRP_GAS1}, the authors study a green PRP for medical nitrous
oxide (N$_2$O) supply chains, integrating cost minimization and GHG emission
reduction. A bi-objective model with a Branch-and-Cut algorithm and Fuzzy
Goal Programming is developed to optimize production, transportation, and
environmental impact. In~\cite{ADMM}, the authors develop an optimization model
for the integrated operation of an Electric Power and Hydrogen Energy System
(IPHS), combining unit commitment, hydrogen production, transportation, and
refueling. They formulate hydrogen transportation as a Vehicle Routing Problem
(VRP) to optimize delivery routes. The overall problem is solved using an
Augmented Lagrangian Decomposition method by dualizing the constraints linking
the hydrogen production and transportation models. In~\cite{IRP_GAS1}, a MILP
model is developed for optimizing distribution and inventory in industrial gas
supply chains, integrating vehicle routing with tank sizing to minimize costs.
However, all these works typically assume fixed inventories. 

In contrast, in this paper we focus on a novel variant of the PRP involving
\emph{mobile inventories}. This problem arises in the context of hydrogen
distribution, where mobile storage units at destinations must be swapped and
transported to other locations.
To the best of our knowledge, studying
this type of PRP, especially in the context of hydrogen transport, is new.


More precisely, the contributions of this paper are the following.
\begin{itemize}
\item We tackle a case study from an industrial partner about a hydrogen delivery problem. 
\item We introduce a new class of problems that we call
  \emph{Production Routing Problem with Mobile Inventories} (PRP-MI).
  As already mentioned, PRP-MI are characterized by the fact that
  mobile storages at destinations can be swapped and transported to other
  locations. Additionally, since these mobile storages can arrive partially filled and depart without reaching full capacity, the refilling process at the sources presents a combinatorial challenge that must be taken into account in the model. In our study, these mobile storages are used for hydrogen distribution, however, the proposed framework can be applied to other transport problems with other commodities and similar logistical challenges. 
  
\item We propose a formulation of a PRP-MI optimization problem on a
  time-expanded graph and present an equivalent Mixed Integer Linear Program
  (MILP) model.
\item We propose a two-step heuristic approach to effectively solve large-scale instances
  of the PRP-MI. First, this heuristic simplifies the problem by omitting
  the combinatorial constraints of the refilling process at the sources to obtain a transport plan for the
  storages. Second, it optimizes the quantities purchased at the sources
  and their allocation to the storages.
\item We conduct comprehensive numerical experiments to compare the performance of three different methods for solving the PRP-MI, evaluating their solution quality across diverse problem instances. 
\end{itemize}

\section{Problem notations and assumptions}
\label{se:notations}
In this section, we turn the problem narrative given in the introduction to a
more formal optimization problem description, and give all assumptions of the problem.

For that purpose, we start by introducing some notations. We denote by
\(\DAYS = \na{1,\ldots, J}\) the set of days of the optimization planning
horizon. Moreover, to follow the problem narrative and as suggested in
Fig.~\ref{fig:two_parts}, we need to refer to the time-intervals before and
after a swap at each destination. For that purpose we introduce two sets related
to day specification, the set $\LDAYS = \nset{ \lday}{\cday \in \DAYS}$ (resp. the
set $\RDAYS = \nset{ \rday}{\cday \in \DAYS}$), where for a day $\cday$, the
notation $\lday$ stands for the first part of the day, when decisions to move
storages from sources are made (resp. $\rday$ stands for the second part of the
day, when deciding to which source to return a storage unit after a swap at
destination).  \optional{To simplify some notations, we also identify symbol
  notation like $\underline{\aday}$ as the value of $j$ by the mapping
  $\underline{(\cdot)}: \DAYS\ni j \mapsto \underline{j} \in \LDAYS$ and the same for
  $\overline{\aday}$.} We then consider the set
$\EDAYS= (\LDAYS \cup \RDAYS) \cup \{\Rday{0}\}$ which is equipped with a finite total
order:
\begin{equation}
  \Rday{0} \preceq \Lday{1} \preceq\Rday{1}\preceq\cdots \preceq \Lday{J} \preceq \Rday{J}
  \eqfinv
  \label{eq:total_order}
\end{equation}
and where $\Rday{0}$ is an extra time index at which the initial conditions of
the problem are given (see Figure~\ref{fig:example_expanded_graph}). For
$i \in \EDAYS$, we denote by $i^+$ the successor (resp.  by $i^{-}$ the
predecessor) of time $i$ if it exists in the total order $(\EDAYS, \preceq)$:
\begin{align*}
  \forall \cday \in \DAYS\eqsepv \np{\Lday{j}}^+
  &= \Rday{j} \eqsepv \np{\Lday{j}}^- = \Rday{j{-}1}\eqfinv
  \\ 
  \forall \cday \in \DAYS\backslash \na{J}
  \eqsepv
  \np{\Rday{j}}^+
  &= \Lday{j{+}1} \eqsepv 
  \\
  \forall \cday \in \DAYS
  \eqsepv
  \np{\Rday{j}}^- &= \Lday{j}
    \eqfinp
\end{align*}

The hours within a day are denoted by $\HOUR=\na{0,\ldots,23}$ and we denote by
$\underline{h}=0$ the smallest (resp. $\overline{h}=23$ the greatest) element of $\HOUR$. 

We also denote by \(\SOURCES\) the set of sources, by \(\STATIONS\) the set of
destinations (also called stations), and by \(\STORAGES\) the set of mobile
storage units. We denote by \(\GRAPH = (\VERTEX, \ARCS)\) the undirected
transportation graph. The set of nodes $\VERTEX$ is the disjoint union of the
sources and destinations, that is \(\VERTEX = \SOURCES \sqcup \STATIONS\). As we assume that
there is no transport between two destinations or two sources, the graph $\GRAPH$ is a
bipartite graph with respect to \(\SOURCES\) and \(\STATIONS\). Without loss
of generality, the graph \(\GRAPH\) is assumed to be complete, meaning
that each destination is reachable from any source and vice-versa with known
transport times that are not necessarily symmetric. Specifically, the transport
time from a source $\source$ to a destination $\destination$, denoted as
$t_{\source,\destination}$, may differ from the transport time from the same
destination $\destination$ back to the same source $\source$, denoted as
$t_{\destination,\source}$. In addition to the transport time between a source
and a destination, we consider a mapping $g: \SOURCES\times\STATIONS \to \RR_{+}$ which
for a given oriented pair $(s,d)$, $g(s,d)$ gives the sum of three terms:
the time needed to load the truck with a storage unit at the source $\source$;
the transport time from $\source$ to $\destination$; the swap time of the
storage units at the destination $\destination$.

For any set $\mathbb A$, we denote by $\findi{\mathbb{A}}(\cdot)$ the indicator
function of $\mathbb{A}$ that equals $1$ when its argument belongs to
$\mathbb{A}$ and $0$ otherwise, by $\Delta_{\mathbb A}$ the diagonal set
$\Delta_{\mathbb A}= \nset{(a,a)}{a\in \mathbb A} \subset A^2$ and, when
$\mathbb A$ is finite, by $\cardinal{\mathbb A}$ its cardinality. Finally, for
all $n\in \NN$, where $\NN$ is the set of natural numbers, we denote the set
$\na{1,\ldots,n}$ by $\IntegerSet{n}$.

The following assumptions are made on the transport model: 
\begin{enumerate}[label=$A_{\arabic*}$]
\item\label{Aun}: All mobile storage units have the same maximal capacity, denoted
  by $\StockMax$;
\item\label{Adeux}: The maximal capacity $\StockMax$ is designed
  to meet at least one day's demand at any destination;
\item\label{Atrois}:  Each destination can receive at most one storage per day;
\item\label{Aquatre}: Each storage is transported at most once per day;
\item\label{Acinq}: Each truck departing from a source \\$\source \in \SOURCES$ starts at a
  given predetermined hour in the morning $\hour_0$;
\item\label{Asix}: The maximum number of storage units that
  can stay at source $\source \in \SOURCES$ is bounded by $\kfun(\source) \in \NN$;
\item\label{Asept}: There is no refill at the sources during the first part of the day;
\item\label{Ahuit}: There is always exactly one storage per destination per day (except during a swap).
\end{enumerate}

As introduced previously, we have two time scales in the transport problem, a
day time scale ($\DAYS$) and an hour time scale ($\HOUR$). We have introduced a
day split in two parts in the previously and we give now the consequence
of this day split at hour level.  Given a source $\source$ from which a storage
unit is sent to a destination $\destination$, the end of the swap at destination
will occur at hour $\hour_0 + g(\source,\destination)$ using
Assumption~\ref{Aquatre} and~\ref{Acinq} and the definition of the mapping $g$.
At each destination $\destination\in \STATIONS$, we divide each day $\aday$
duration into two time periods: the first part of the day which is the time
period $[\underline{h},\SwapV{\aday}{\source}{\destination}[$ from midnight
($\underline{h}$) to $\SwapV{\aday}{\source}{\destination}$, where
$\SwapV{\aday}{\source}{\destination}$ denotes the hour at which the swap ends
at destination $\destination$ at day $\aday$ and the last, or second, part of
the day which is the time period
$[\SwapV{\aday}{\source}{\destination},\overline{h}+1[$ from the end of swap
until midnight ($\overline{h}+1$).  Indeed, using the definition of the function
$g$ and of time $h_0$, we have that
$\SwapV{\aday}{\source}{\destination} = \hour_0 + g(\source,\destination)$,
where $\source$ is the source from which the storage unit was transported to
$\destination$ on day $\aday$.  Together with Assumption~\ref{Atrois}, during a
fixed day, we obtain that at most two distinct storages are used to fulfill the
demand, one for each time period. Finally, if during day $\aday$ no storage is
sent to destination~$\destination$, there is no swap time at destination and we
arbitrarily divide the day into two time periods, which will be explained in
more details later in Equation~\eqref{eq:forced_swap_time}. An example of a swap
is illustrated in Fig.~\ref{fig:hympulsion_swap}.

\begin{figure*}[t]
\begin{align*}
  s \to \overbrace{\includegraphics[scale=.15]{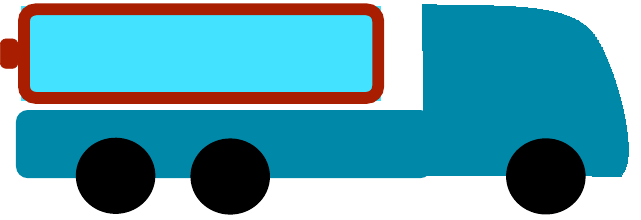}}^{\text{from $s$ to $d$}} \to d \hspace{2cm}
  \underbrace{\includegraphics[scale=0.15]{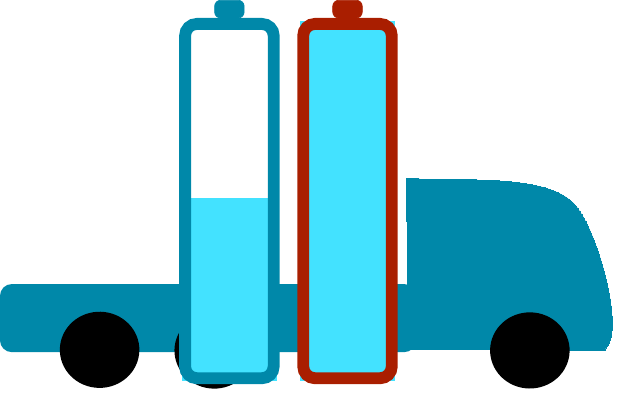}}_{\text{swap at $d$}}
  & \hspace{2cm}%
    d \to \overbrace{\includegraphics[scale=0.15]{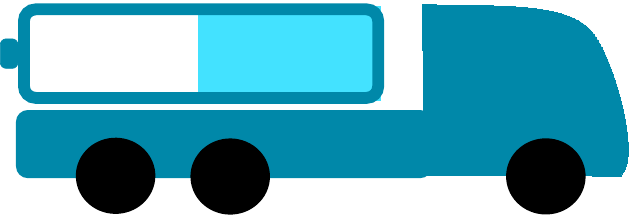}}^{\text{from $d$ to $s'$}} \to s'
  \\
  \underbrace{\mbox{\includegraphics[height=1.5cm]{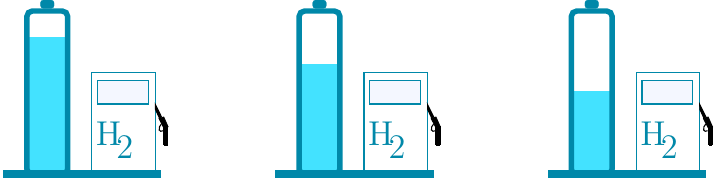}}\hspace{1.5cm}}%
  _{\text{$d$: first part of the day $[0,\SwapV{\aday}{\source}{\destination}[$}}
  &
    \underbrace{\mbox{\includegraphics[height=1.5cm]{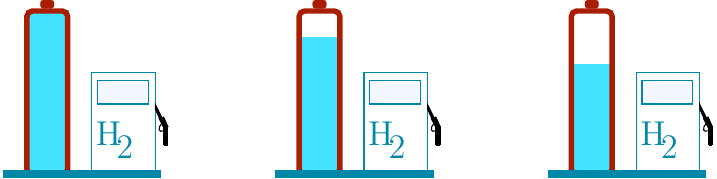}}}%
    _{\text{$d$: second part of the day $[\SwapV{\aday}{\source}{\destination},\overline{h}[$}}
\end{align*}
\caption{The day division at a destination $d$: a truck picks up a storage unit
  from the source $s$ at \(\hour_0\), transports it to the destination $d$,
  performs the swap, and then returns the collected storage unit back to a
  source $s'$.  During the first part of the day, the (blue) mobile storage
  originally located at destination $d$ satisfies the demand, and during the
  second part of the day, the newly arrived mobile storage (red) fulfills the
  demand}
\label{fig:hympulsion_swap}
\end{figure*}

\section{Problem description on a time-expanded flow graph}
\label{se:problem_description_expanded_graph}

In the context of routing problems with a fleet of vehicles, previous studies
(e.g.,~\cite{comparison_formulation_IRP1,comparison_formulation_IRP2}) have
demonstrated that vehicle-indexed formulations yield high-quality solutions when
considering complete tours and a limited number of vehicles (typically fewer
than five). However, our problem differs in that the routes follow a
back-and-forth structure: storage units depart from the source, reach a
destination, and the storage unit previously stationed at that destination
immediately returns, rather than forming complete multi-stop tours.

Moreover, in our case, it is the mobile storage units that are transported
between locations and that require the monitoring of their stock, making a
storage-indexed formulation more appropriate rather than a vehicle-indexed
formulation. This approach, however, poses a significant numerical challenge as
we expect our model to handle a large number of storage units, and using a
storage-indexed formulation leads to an excessive number of variables. In a preliminary study (not described here), we experimented with such formulation but encountered severe
computational difficulties due to the presence of symmetrical solutions. For
example, if two identical full storage units are available at a source for
dispatch, interchanging them does not alter the solution structure or cost, yet
it introduces redundancies that significantly impact solver performance. This
issue is further exacerbated as the number of mobile storage units increases, and
in the absence of appropriate symmetry-breaking constraints. We refer
to the study in~\cite{Margot2010} for the analysis of MILP models in the presence
of symmetries.

To keep the number of variables reasonable and to avoid such symmetries, we propose a \emph{time-expanded} graph reformulation of our
problem. This reformulation eliminates the need for mobile storage indexing,
thereby breaking some symmetries and is expected to improve numerical
resolution. For that purpose, we build from the transport graph a new graph
commonly called~\cite{time_expended_network,time_expended_network2} a
\emph{time-expanded graph} (or graph dynamic network flow). The transport of
storage unit is replaced by considering both continuous and binary flow
variables attached to the arcs of the time-expanded graph and satisfying flow
conservation equations (Kirchhoff law) at the graph nodes (to be explained
later).


In Sect.~\ref{sec:flow_to_planning}, we demonstrate how the storage routes (transport
planning) can be derived from a solution of the problem on this time-expanded
flow graph.

\subsection{Definition of the time-expanded graph}

We turn now to the precise description of the time-expanded graph. In~\S\ref{sbsec:def_nodes} we define the nodes of the time-expanded graph and in~\S\ref{sbsec:def_arcs}, we define its arcs.

\subsubsection{Definition of nodes}\label{sbsec:def_nodes}
First, we define the set of nodes of the time-expanded graph.
Each node of the expanded graph is an ordered pair composed of a node of the physical
transportation graph (source or destination), that is an element of \(\VERTEX\),
and of a time index belonging to $\EDAYS = \LDAYS \cup \RDAYS \cup \{\Rday{0}\}$.  Thus,
the set of nodes of the time-expanded graph, denoted as $\VERTEX^{\TEG}$,
is defined by $\VERTEX^{\TEG}=\VERTEX\times\EDAYS$.

\subsubsection{Definition of arcs}
\label{sbsec:def_arcs}
Second, we define the arcs of the time-expanded graph. The time-expanded graph
we build is an oriented multigraph, that is, each edge in the graph is oriented,
and two nodes may be connected by more than one oriented edge. As it is used to
model physical flows between nodes, an arc is only present between two nodes of
$\VERTEX^{\TEG}$ indexed by two consecutive elements of the time index $\EDAYS$,
that is, between a node $(v,i)$ and a node $(v', i^+)$ where $i^+$ is the
successor of $i\in\EDAYS$ (see Equation~\eqref{eq:total_order}). For that reason,
we only use one time index in an arc specification, that is, the arc
$(v, i) \mapsto (v', i^+)$ in the time-expanded graph is denoted by $(v,v', i)$ and
$\varphi_{v,v',i}$ when an arc is used as a flow index and a specific notation when
two nodes are connected by several arcs as described now.

Since at most one storage unit is transported
between a source and a destination (and vice-versa) using
Assumption~\ref{Atrois}, and exactly one storage is present at each destination
and each time using Assumption~\ref{Ahuit}, a single arc
$(v,v',i)=(v, i) \mapsto (v', i^+)$ is sufficient to represent the flow of hydrogen
between two nodes $(v,i)$ and $(v',i^+)$ of the time-expanded graph for
$(v,v') \in (\SOURCES{\times}\STATIONS) \cup \Delta_{\STATIONS}$ or
$(v,v') \in (\STATIONS{\times}\SOURCES) \cup \Delta_{\STATIONS}$, where
$\SOURCES \times \STATIONS$ and $\STATIONS \times \SOURCES$ correspond to the arcs of the
physical graph $\GRAPH$, and, the addition of $\Delta_{\STATIONS}$ ensures that a
storage located at a destination can remain at that destination, while also
preventing its transport to a different destination. However, the situation is
slightly more complicated when the two involved physical nodes are the same
source node, that is $v=v'$ with $v\in \SOURCES$.  As multiple storage units may
remain at a source without being dispatched to a destination, we need several
arcs between $(\source,i)$ and $(\source,i^+)$ for $s\in \SOURCES$. This number of
arcs is bounded for each source $s\in \SOURCES$ by the value $\kfun(s)$ where the
function $\kfun$ is given by Assumption~\ref{Asix}. Finally, for
$k\in \IntegerSet{\kfun(\source)}$, the notation $\lTEGarc{i}{\source}{k}$ denote
the $k$-th arc $(\source,i) \mapsto (\source,i^+)$. We also use the
notation $\varphi_{\source: k , i}$ when a $k$-th arc from
a source is used as a flow index.

To conclude, the time-expanded graph is defined by
$\GRAPH^{\TEG} = (\VERTEX^{\TEG}, \ARCS^{\TEG})$, where the set of nodes
$\VERTEX^{\TEG}$ is $\VERTEX \times \EDAYS$ and the set of arcs $\ARCS^{\TEG}$ is
formally defined as follows:
\begin{subequations}
\begin{align}
  \ARCS^{\TEG}
  &= \LARCS \cup \RARCS \cup \SARCS
  \\
  {\text{with }}&\LARCS
  = \bp{(\SOURCES{\times}\STATIONS) \cup \Delta_{\STATIONS}}\times \LDAYS
  \eqfinv\label{eq:set_arc_first_part}
  \\
  &\RARCS
    = \bp{(\STATIONS{\times}\SOURCES) \cup \Delta_{\STATIONS}}\times (\RDAYS\cup \{\Rday{0}\})
    \eqfinv\label{eq:set_arc_second_part}
  \\
  &\SARCS
    =  \bset{(\source{:}k)}{s\in \SOURCES, k \in \IntegerSet{\kfun(s)}}
    {\times} \EDAYS
    \eqfinp\label{eq:set_arc_source}
\end{align}
\end{subequations}

The set $\LARCS$ defined in Equation~\eqref{eq:set_arc_first_part} represents
arcs whose first node has its time index in $\LDAYS$, excluding arcs that
connect sources to themselves. The set $\RARCS$ defined in
Equation~\eqref{eq:set_arc_second_part} represents arcs whose first node has
its time index in $\RDAYS \cup \{\overline{0}\}$, excluding arcs that connect sources to
themselves. The set $\SARCS$ defined in Equation~\eqref{eq:set_arc_source}
represents arcs that connect each source to itself.

\begin{definition}\label{def:arc_properties}
For an arc
$\arc \in \ARCS^{\TEG}$, we denote by $\tailarc{\arc} \in \VERTEX^{\TEG}$ its tail
node, by $\headarc{\arc} \in \VERTEX^{\TEG}$ its head node and by
$\dayarc{\arc} \in \EDAYS$ the time index of its tail node.
\end{definition}

We show in Fig.~\ref{fig:example_arc_expanded_graph} an example of a
time-expanded graph with one source and two destinations, where the arcs
belonging to $\LARCS$ are in green, to $\RARCS$ are in red and to $\SARCS$ are
in blue.

\begin{figure}[hbtp]
  \begin{center}
    \includegraphics[width=0.45\textwidth]{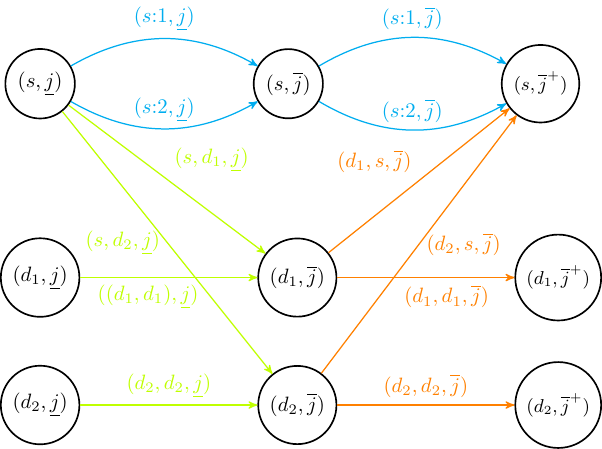}
  \end{center}
  \caption{Graphic display of a time-expanded graph with one source -- where the maximum number of unit storages
    at source is 2 --- and two destinations. 
    The arcs belonging to $\LARCS$ are in green, to $\RARCS$ are in orange and to $\SARCS$ are in blue}
  \label{fig:example_arc_expanded_graph}
\end{figure}

\begin{figure*}
  \begin{center}
    \includegraphics[width=0.8\textwidth]{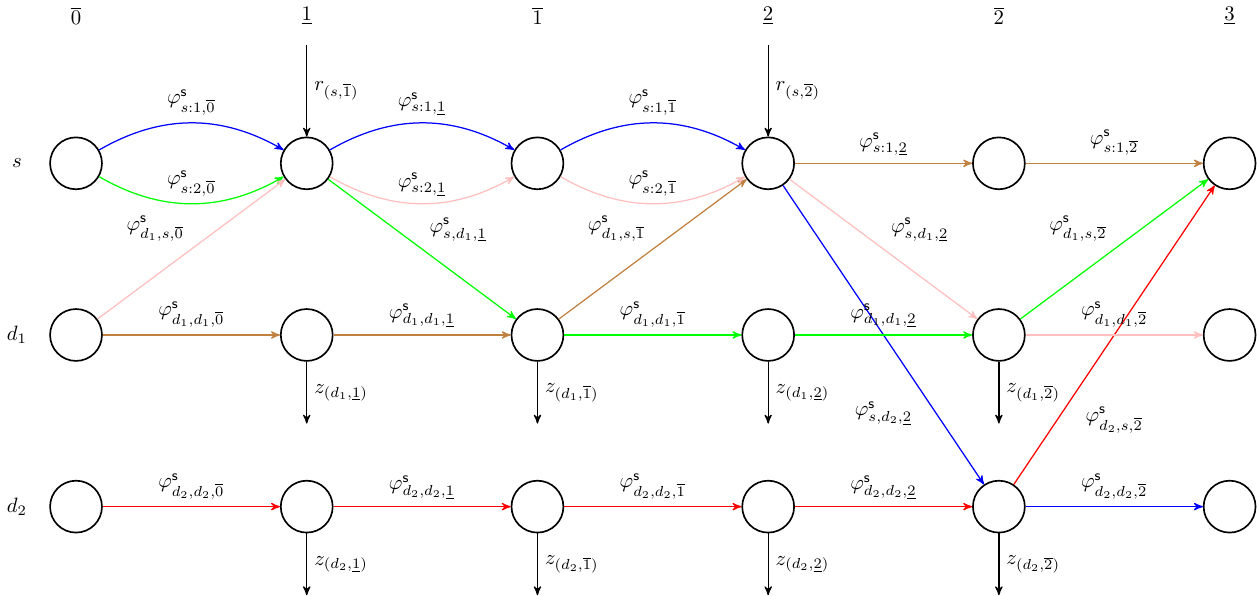}
  \end{center}
  \caption{Display on a time-expanded graph (with one source, two destinations
    and an horizon of two days) of a storage flow solution $\Storageflow$.
    We only display the subset of arcs where the
    storage flow is equal to one. As developed in Sect.~\ref{sec:flow_to_planning},
    transport planning, represented here arc paths (sequence of arcs with the same color),
    can be derived from $\Storageflow$.
    \label{fig:example_expanded_graph}}
\end{figure*}

\subsection{Flows in the time-expanded graph}\label{flow_model_description}

On the time-expanded graph $\GRAPH^{\TEG} = (\VERTEX^{\TEG}, \ARCS^{\TEG})$, we
mainly consider two flows. First, a continuous flow, that we call \emph{hydrogen
  flow}, which is a function $\Hydrogenflow: \ARCS^{\TEG} \to [0,\StockMax]$ which
is used to quantify the quantity of hydrogen transported along the arcs, where
$\StockMax$ represents the maximum possible hydrogen flow, which also
corresponds to the maximum capacity of a storage $\StockMax$ given by
Assumption~\ref{Aun}. Second, a binary flow, that we call \emph{storage flow},
which is a function $\Storageflow: \ARCS^{\TEG} \to \na{0,1}$ which is used to
detect if a storage is used to transport hydrogen along an arc. For simplicity,
for any arc $a \in \ARCS^{\TEG}$, we will denote by $\Hydrogenflow_{a}$ (resp. $\Storageflow_{\arc}$) the value returned by
the function $\Hydrogenflow$ (resp. $\Storageflow)$ for the same arc, that is $\Hydrogenflow_a=\Hydrogenflow(a)$ (resp.  $\Storageflow_a=\Storageflow(a)$). The relationship between
the continuous and binary flow will be established in
Equation~\eqref{eq:flow_arc}.

\section{Problem modeling on graph $\GRAPH^{\TEG}$}
\label{sec:modeling}
In this section, we model an optimization transport problem on the time-expanded
graph and we explicit the constraints on the continuous and binary flow
variables at the nodes and on the arcs of the time-expanded graph.

\subsection{Decision variables}
The decision variables are composed on the one hand of the flow decision variables introduced in~\S\ref{flow_model_description} and described in Table~\ref{table:flow_decisions}, and on the other hand of the node decision variables, which are the refilling decisions at the sources and the emptying and swap decisions at the destinations, and which are described in Table~\ref{table:node_decisions}.
\begin{table*}[t]
  \begin{center}
  \begin{tabular}{|c|c|c|}
    \hline 
    Decision
    &Domain
    & Description\\ \hline \hline
    $\Hydrogenflow = \nseqa{\Hydrogenflow_a}{a\in \ARCS^{\TEG}}$
    &$[0,\StockMax]^{|\ARCS^{\TEG}|}$
    & $\Hydrogenflow_a$ is the hydrogen flow in kg
    \\
    && circulating on the arc $a \in \ARCS^{\TEG}$
    \\ \hline
    $\Storageflow= \nseqa{\Storageflow_a}{a\in \ARCS^{\TEG}}$
    &$\na{0,1}^{|\ARCS^{\TEG}|}$
    &     $\Storageflow_a$ is  the storage flow circulating
    \\ &&
          on the arc $a \in \ARCS^{\TEG}$. \\ \hline
  \end{tabular}
  \caption{Hydrogen and storage flow variables\label{table:flow_decisions}}
\end{center}
\end{table*}

\begin{table*}[t]
  \begin{center}
  \begin{tabular}{|c|c|c|}
    \hline 
    Decision
    &Domain
    & Description
    \\ \hline\hline
    $z= \nseqa{\nodeV{z}{i}{\destination}}{(\destination,i)\in \STATIONS{\times}(\LDAYS \cup \RDAYS)}$
    &$\nodeV{z}{i}{\destination} \in [0,\StockMax]$
    &$\nodeV{z}{i}{\destination}$ is the quantity of hydrogen in kg used 
    \\
    &
    & to satisfy demand at node $(\destination, i)$
    \\ \hline
    $\swap=\nseqa{\SwapV{\aday}{\source}{\destination}}%
    {(\source,\destination,\aday)\in \SOURCES{\times}\STATIONS{\times}\DAYS}
    $
    &$\SwapV{\aday}{\source}{\destination}\in \HOUR$
    &$\SwapV{\aday}{\source}{\destination}$ is the hour at which swap happens at
    \\
    & & destination $\destination$ during day $\aday \in \DAYS$
    \\
    & & for a storage coming from $\source$
    \\
    & & (arbitrary when no swap)
    \\ \hline
    $r= \nseqa{\nodeV{r}{\Lday{\aday}}{\source}}{(\source,\Lday{\aday})\in \SOURCES{\times}\LDAYS}$
    &$\nodeV{r}{\Lday{\aday}}{\source}\in [0,\overline{r}_\source]$
    &$\nodeV{r}{\Lday{\aday}}{\source}$ is the quantity of hydrogen in kg used
    \\
    & &  to fill the storages at node  $(\source,\Lday{\aday})$
    \\ \hline
    $\sigma= \nseqa{\bijection{\Lday{\aday}}{\source}}{(\source,\Lday{\aday})\in \SOURCES{\times}\LDAYS}$
    && $\bijection{\Lday{\aday}}{\source}$ is a bijection of $[|\STATIONS| + \kfun(\source)]$ used to 
    \\
    && ``follow'' storages at sources (see~\eqref{eq:in_out_flow_source}-\eqref{eq:permutation_constraints})
    \\ \hline
  \end{tabular}
  \caption{Node variables\label{table:node_decisions}}
  \end{center}
\end{table*}

\subsection{Arc and node constraints}

The initial position of the storages on the time-expanded graph $\GRAPH^{\TEG}$
are part of the data of the optimization problem.  They are given by selecting
the active storage flow at time index
$\Rday{0}$.  Therefore, the active storage flow at time index $\Rday{0}$ must be
equal to the cardinality of the storages set, that is,
\begin{equation}\label{eq:initial_position_storage_set}
  \sum_{ \nset{ \arc \in\ARCS^{\TEG}}{\dayarc{a}=\Rday{0}}} \Storageflow_a
  = \cardinal{\STORAGES}
  \eqfinp
\end{equation}
Here, $\dayarc{\arc}$ denotes the time index of the starting node of $\arc$, as introduced in Definition~\ref{def:arc_properties}.
\subsubsection{Arc constraints}
\noindent$\bullet$ If there is a positive hydrogen flow on an
arc $a \in \ARCS^{\TEG}$, that is, if $\Hydrogenflow_a >0$, then a storage must be
present on the same arc to realize the transport, that is, $\Storageflow_a =
1$. Thus the two flows $\Hydrogenflow$ and $\Storageflow$ are linked by the
following relation
\begin{subequations}\label{eq:flow_arc}
\begin{align}
 \forall a \in \ARCS^{\TEG} \eqsepv
  \Hydrogenflow_a >0 \implies \Storageflow_a =1 \label{flow_arc_1} \eqfinp
\intertext{If no storage is sent on an arc then, there is
no hydrogen flow transported on the same arc, that is,}
  \forall a \in \ARCS^{\TEG} \eqsepv 
  \Storageflow_a = 0  \implies \Hydrogenflow_a =0
  \eqfinp \label{flow_arc_bis} 
\end{align}
\end{subequations}
Note that the converse implication does not
hold. Indeed, it is possible --- and important for returning empty unit storages
to sources --- for a storage to be transported on an arc, $\Storageflow_a =1$
with an empty stock $\Hydrogenflow_a = 0$.

\noindent$\bullet$ For all $\destination \in \STATIONS$ and for all
$\aday \in \EDAYS$, there must always be exactly one storage at each destination
(see Assumption~\ref{Ahuit})
\begin{equation}\label{eq:always_storage_destination}
  \SFlow{\aday}{\destination}{\destination}
  = 1 \eqfinp 
\end{equation}

\subsubsection{Constraints at nodes in $\STATIONS{\times}\LDAYS$}
\noindent$\bullet$ For all $\destination \in \STATIONS$ and all
$\aday \in \DAYS$, the hydrogen outflow
$\HFlow{\Lday{\aday}}{\destination}{\destination}$ at node
$(\destination,\Lday{\aday})$, is equal to the hydrogen inflow
$\HFlow{\Lday{\aday}^{-}}{\destination}{\destination}$ at the same node, minus
the quantity $\nodeV{z}{\Lday{\aday}}{\destination}$ used to satisfy the demand
during the first part of the day (flow conservation), that is, for all
$\destination \in \STATIONS$ and $\aday \in \DAYS$
\begin{equation}
  \label{eq:demand_satisfaction1}
  \HFlow{\Lday{\aday}}{\destination}{\destination} =
  \HFlow{\Lday{\aday}^-}{\destination}{\destination}-
  \nodeV{z}{\Lday{\aday}}{\destination}  \eqfinp
\end{equation}

\noindent$\bullet$ For all $\destination \in \STATIONS$ and $\aday \in \DAYS$, the total
demand $\nodeV{\Demand}{\Lday{\aday}}{\destination}$ at node
$(\destination,\Lday{\aday})$ (total demand at day $\aday$ before the swap time), is given by
\begin{subequations}
\begin{align}
  \nodeV{\Demand}{\Lday{\aday}}{\destination}
  &= \sum\limits_{\hour \in \HOUR} \findi{\na{\hour}}(\SwapV{\aday}{\source}{\destination})
  \times C_{\destination,\aday,\hour}\label{eq:demand_destination1} \eqfinv
\intertext{with}
 C_{\destination,\aday,\hour}&= \sum\limits_{\hour'=0}^{\hour}q_{\demand,\aday,\hour'} \eqfinv
\end{align}
\end{subequations}
where $q_{\demand,\aday,\hour}$ is the demand at hour
$\hour$ of day $\aday$ at destination $\destination$ and
$C_{\destination,\aday,\hour}$ is the cumulative demand from midnight (0) to
hour $\hour$ of day $\aday$ at destination $\destination$.

\medskip
\noindent$\bullet$ For all $\destination \in \STATIONS$ and $\aday \in \DAYS$, the
quantity $\nodeV{z}{\Lday{\aday}}{\destination}$ used to satisfy the demand at
node $(\demand,\Lday{\aday})$, is equal to the minimum between the inflow at the
same node and the demand, that is,
\begin{equation}\label{eq:emptying_bounded1}
  \nodeV{z}{\Lday{\aday}}{\destination}
  =\min( \HFlow{\Lday{\aday}^-}{\destination}{\destination}, \nodeV{\Demand}{\Lday{\aday}}{\destination}) \eqfinp
\end{equation}
This ensures that the quantity  $\nodeV{z}{\Lday{\aday}}{\destination}$ to satisfy the demand is equal to the demand if there is enough hydrogen, otherwise, it is equal to the available hydrogen inflow.

\subsubsection{Constraints at nodes in $\STATIONS{\times}\RDAYS$}
\label{sub:node_rday}
\noindent$\bullet$ For all $\destination \in \STATIONS$ and for all
$\aday \in \DAYS$, a node $(\destination,\Rday{\aday})$ can have at most one
positive storage inflow originating from a source node, that is,
\begin{equation}\label{eq:at_most_one_storage_destination}
  \sum\limits_{\source \in \SOURCES} \SFlow{\Rday{\aday}^-}{\source}{\destination} \leq 1 \eqfinp
\end{equation}

\noindent$\bullet$ For all $\destination \in \STATIONS$ and $\aday \in \DAYS$, since the
number of positive storage inflows at node $(\destination,\Rday{\aday})$ must
match the number of positive storage outflows at the same node, it follows that
the number of positive storage inflows originating from a source node is equal
to the number of positive storage outflows going to a source node (flow
conservation), that is,
\begin{equation}\label{eq:storage_in_storage_out_destination}
  \sum\limits_{\source \in \SOURCES}
  \SFlow{\Rday{\aday}^-}{\source}{\destination} = \sum\limits_{\source
    \in \SOURCES} \SFlow{\Rday{\aday}}{\destination}{\source} \eqfinp
\end{equation}

\noindent$\bullet$ For all $\destination \in \STATIONS$ and $\aday \in \DAYS$, if there is
a storage inflow originating from a source at node
$(\destination,\Rday{\aday}^-)$, then the hydrogen inflow
$\HFlow{\Rday{\aday}^-}{\destination}{\destination}$ will depart to a source node,
that is,
\begin{align}\label{eq:swap_destination}
  \HFlow{\Rday{\aday}^-}{\destination}{\destination} =
  \findi{\na{0}}\bp{\sum\limits_{\source \in \SOURCES} \SFlow{\Rday{\aday}^-}{\source}{\destination}}
  \HFlow{\Rday{\aday}^-}{\destination}{\destination} + \sum\limits_{\source \in \SOURCES}
  \HFlow{\Rday{\aday}}{\destination}{\source} \eqfinp
\end{align}

\noindent$\bullet$ For all $\destination \in \STATIONS$ and $\aday \in \DAYS$, the
hydrogen outflow at node $\TEGnode{\Rday{\aday}}{\destination}$,
$\HFlow{\Rday{\aday}}{\destination}{\destination}$, is given by
\begin{align}\label{eq:demand_satisfaction2}
  \HFlow{\Rday{\aday}}{\destination}{\destination}
  =
  \sum\limits_{\source \in \SOURCES} \HFlow{\Rday{\aday}^-}{\source}{\destination}
  &+ \findi{\na{0}}
  \bp{\sum\limits_{\source \in \SOURCES} \SFlow{\Rday{\aday}^-}{\source}{\destination}}
  \HFlow{\Rday{\aday}^-}{\destination}{\destination}  \nonumber
 \\  &- \nodeV{z}{\Rday{j}}{\destination} \eqfinp
\end{align}
Indeed, if there is a positive hydrogen inflow at node
$\TEGnode{\Rday{\aday}}{\destination}$ from a node $(\source',\Rday{\aday}^-)$,
we have that
$\sum\limits_{\source \in \SOURCES} \SFlow{\Rday{\aday}^-}{\source}{\destination}=1$
(as the sum is bounded using
Equation~\eqref{eq:at_most_one_storage_destination}), which means that a swap
happens, that is,
$\HFlow{\Rday{\aday}}{\destination}{\destination} = \sum\limits_{\source \in
  \SOURCES} \HFlow{\Rday{\aday}^-}{\source}{\destination}-
\nodeV{z}{\Rday{j}}{\destination}=
\HFlow{\Rday{\aday}^-}{\source'}{\destination} -
\nodeV{z}{\Rday{j}}{\destination}$.  Otherwise, if
$\sum\limits_{\source \in \SOURCES} \SFlow{\Rday{\aday}^-}{\source}{\destination}=0$,
we have, using Equation~\eqref{flow_arc_bis}, that
$\sum\limits_{\source \in \SOURCES} \HFlow{\Rday{\aday}^-}{\source}{\destination}=0$
and that
$\HFlow{\Rday{\aday}}{\destination}{\destination}=
\HFlow{\Rday{\aday}^-}{\destination}{\destination} -
\nodeV{z}{\Rday{j}}{\destination}$.

\noindent$\bullet$ The swap time at node $(\destination,\Rday{\aday})$ equals the
departure time $\hour_0$ of the storage flow plus the travel time from the
source node to the destination node, that is, for all $\source \in \SOURCES$ for
all $\destination \in \STATIONS$ and for all $\aday \in \DAYS$, we have
\begin{equation}\label{eq:swap_time}
  \SFlow{\Rday{\aday}^-}{\source}{\destination} =  1
  \implies \SwapV{\aday}{\source}{\destination} = \hour_0 + g(\source,\destination)
  \eqfinp
\end{equation}

\noindent$\bullet$ If no swap occurs at the node, the division of the day $\aday$ into
its first part $\Rday{\aday}^-$ and second part $\Rday{\aday}$ at the same node
is made arbitrarily. To reduce symmetry, we set the swap variable to the 
predetermined value $(\overline{h}+\underline{h}+1)/2$, that is, for all
$\destination \in \STATIONS$ and $\aday \in \DAYS$ we have
\begin{equation}\label{eq:forced_swap_time}
  \sum\limits_{\source \in \SOURCES} \SFlow{\Rday{\aday}^-}{\source}{\destination}
  =  0 \implies \SwapV{\aday}{\source}{\destination} =
  \frac{\overline{h}+\underline{h}+1}{2}
  \eqfinp 
\end{equation}

\noindent$\bullet$ For all $\destination \in \STATIONS$ and $\aday \in \DAYS$, the total
demand $\nodeV{\Demand}{\Rday{\aday}}{\destination}$ at node $(\demand,\Rday{\aday})$ is given by
\begin{equation}\label{eq:demand_destination2}
  \nodeV{\Demand}{\Rday{\aday}}{\destination}=
  C_{\aday,\destination,\overline{h}} - \nodeV{\Demand}{\Rday{\aday}^-}{\destination} \eqfinv
\end{equation}
where $ C_{\aday,\destination,\overline{h}}$ is the total demand of day $\aday$
at destination $\destination$. From Assumption~\ref{Adeux}, we have that
\[C_{\aday,\destination,\overline{h}} \leq \StockMax \eqfinp \]

\noindent$\bullet$ For all $\destination \in \STATIONS$ and $\aday \in \DAYS$, the
quantity $\nodeV{z}{\Rday{\aday}}{\destination}$ used to satisfy the demand at
node $(\demand,\Rday{\aday})$, is equal to the minimum between the inflow at the
same node and demand, that is,
\begin{equation}\label{eq:emtying_bounded2}
  \nodeV{z}{\Rday{\aday}}{\destination}
  =\min( \HFlow{\Rday{\aday}^-}{\destination}{\destination}, \nodeV{\Demand}{\Rday{\aday}}{\destination})  \eqfinp
\end{equation}

\subsubsection{Constraints at nodes in $\SOURCES {\times} \LDAYS$}\label{constraints_node_source_1}

\noindent$\bullet$ For all $\source \in \SOURCES$ and for all $\aday \in \DAYS$, the
number of positive storage inflows at node $(\source,\Lday{\aday})$ is bounded
by the maximum number of storage units that can stay at the source, that is,
\begin{equation}\label{eq:limit_storage_source}
  \sum\limits_{\destination \in \STATIONS} \SFlow{\Lday{\aday}^-}{\destination}{\source}
  + \sum\limits_{k \in \cardinal{\kfun(\source)}} \SFlow{\Lday{\aday}^-}{\source}{k} \leq \kfun(\source) \eqfinp
\end{equation}

\noindent$\bullet$ For all $\source \in \SOURCES$ and $\aday \in \DAYS$, the number of
positive storage inflows at node $(\source,\Lday{\aday})$ is equal to the number
of positive storage outflows at the same node (flow conservation), that is,
\begin{align} \label{eq:storage_sent_leaving_source}
  \sum\limits_{\destination \in \STATIONS}
  &\SFlow{\Lday{\aday}^-}{\destination}{\source}
    + \sum\limits_{k \in\IntegerSet{\kfun(\source)}} \SFlow{\Lday{\aday}^-}{\source}{k}
    = \sum\limits_{\destination \in \STATIONS} \SFlow{\Lday{\aday}}{\source}{\destination}
    \nonumber \\
  &+ \sum\limits_{k \in\IntegerSet{\kfun(\source)}}
    \SFlow{\Lday{\aday}}{\source}{k}
    \eqfinv
\end{align}
the initial quantities
$\{\SFlow{\Rday{0}}{\destination}{\source}\}_{\destination \in \STATIONS}$ and
$\{\SFlow{\Rday{0}}{\source}{k}\}_{k \in\IntegerSet{\kfun(\source)}}$ being given.
        
\noindent$\bullet$ For all $\source \in \SOURCES$ and $\aday \in \DAYS$, the hydrogen flow
conservation equation is satisfied at each node
$\TEGnode{\Lday{\aday}}{\source}$. The hydrogen inflow plus the refill (which is
the quantity of hydrogen purchased at the node), that is, the left-hand side of
Equation~\eqref{eq:flow_conservation_source}, is equal to the hydrogen outflows
at the same node, that is, the right-hand side of
Equation~\eqref{eq:flow_conservation_source}
\begin{align}\label{eq:flow_conservation_source}
  \sum\limits_{\destination \in \STATIONS}
  &\HFlow{\Lday{\aday}^-}{\destination}{\source}
    + \sum\limits_{k \in\IntegerSet{\kfun(\source)}} \HFlow{\Lday{\aday}^-}{\source}{k}
    + \nodeV{r}{\Lday{\aday}}{\source} \nonumber
  \\
  &= \sum\limits_{\destination \in \STATIONS} \HFlow{\Lday{\aday}}{\source}{\destination}
    + \sum\limits_{k \in\IntegerSet{\kfun(\source)}} \HFlow{\Lday{\aday}}{\source}{k} \eqfinv
\end{align}
the initial quantities 
$\{\HFlow{\Rday{0}}{\destination}{\source}\}_{\destination \in \STATIONS}$ and
$\{\HFlow{\Rday{0}}{\source}{k}\}_{k \in\IntegerSet{\kfun(\source)}}$ being given.

\noindent$\bullet$ For all $\source \in \SOURCES$ and $\aday \in \DAYS$, the quantity of
hydrogen purchased at node $(\source,\Lday{\aday})$ is bounded, that is,
\begin{equation}\label{eq:refill_bounded}
  \nodeV{r}{\Lday{\aday}}{\source}
  \leq \overline{r}_\source \eqfinp
\end{equation}
\noindent$\bullet$ For each $\source \in \SOURCES$ and each $\aday \in \DAYS$, the stock
in all storage units present at source $\source$ must not decrease during
the refilling process.  Thus, there must exist a bijection which pairs the
hydrogen inflows and outflows at source node $\TEGnode{\Lday{\aday}}{\source}$
respecting a non decreasing flow constraint. More formally, to encode this
constraint, we fix a source $\source \in \SOURCES$ and a day $\aday \in \DAYS$
and start by defining two vectors of hydrogen and storage inflows
$\{\Flow{f^{\mathsf{in}}}{\Lday{\aday}}{\source}{l} \}_{l \in \MaxArcsSource }$
and
$\{\Flow{y^{\mathsf{in}}}{\Lday{\aday}}{\source}{l} \}_{l \in \MaxArcsSource}$,
where
\begin{equation}
  \MaxArcsSource= \IntegerSet{|\STATIONS| + \kfun(\source)}
  \eqfinv
  \label{eq:MaxArcsSource}
\end{equation}
and two vectors of hydrogen and storage outflows
$\{\Flow{f^{\mathsf{out}}}{\Lday{\aday}}{\source}{l} \}_{l \in \MaxArcsSource
}$ and $\{\Flow{y^{\mathsf{out}}}{\Lday{\aday}}{\source}{l} \}_{l \in
  \MaxArcsSource }$ by the following equations
\begin{subequations} 
  \label{eq:in_out_flow_source}
\begin{align}
  \Flow{f^{\mathsf{in}}}{\Lday{\aday}}{\source}{l}
  &= 
    \begin{cases} 
      \HFlow{\Lday{\aday}^-}{d_l}{\source},  \text{ if } l \leq |\mathbb{D}| \eqfinv\\ 
      \HFlow{\Lday{\aday}^-}{\source}{l - |\STATIONS|},  \text{ otherwise} \eqfinv
    \end{cases}
  \\
  \Flow{f^{\mathsf{out}}}{\Lday{\aday}}{\source}{l}
  &= 
    \begin{cases} 
      \HFlow{\Lday{\aday}}{\source}{d_{l}},  \text{ if } l \leq |\mathbb{D}| \eqfinv \\ 
      \HFlow{\Lday{\aday}}{\source}{l - |\STATIONS|},  \text{ otherwise} \eqfinv
    \end{cases}
  \\
  \label{def:yin}
  \Flow{y^{\mathsf{in}}}{\Lday{\aday}}{\source}{l}
  &= 
    \begin{cases} 
      \SFlow{\Lday{\aday}^-}{d_l}{\source},  \text{ if } l \leq |\mathbb{D}| \eqfinv\\ 
      \SFlow{\Lday{\aday}^-}{\source}{l - |\STATIONS|},  \text{ otherwise} \eqfinv
    \end{cases}
  \\
  \label{def:yout}
  \Flow{y^{\mathsf{out}}}{\Lday{\aday}}{\source}{l}
  &= 
    \begin{cases} 
      \SFlow{\Lday{\aday}}{\source}{d_{l}},  \text{ if } l \leq |\mathbb{D}| \eqfinv \\ 
      \SFlow{\Lday{\aday}}{\source}{l - |\STATIONS|},  \text{ otherwise} \eqfinv
    \end{cases}  
\end{align}
\end{subequations}
for all $l \in \MaxArcsSource$, where $\destination_l$ is the $l$-{th} destination
in the set $\STATIONS$. Now, we formally define the nondecreasing stock
constraint. Let $\PermutationBijectionDomain_{\source,\lday}$ be the set of
active incoming arcs at node $\TEGnode{\Lday{\aday}}{\source}$, that is
$\PermutationBijectionDomain_{\source,\lday} = \nset{l}{
  \Flow{y^{\mathsf{in}}}{\Lday{\aday}}{\source}{l}=1 }$, and let
$\PermutationBijectionCoDomain_{\source,\lday}$ be the set of active outgoing
arcs at node $\TEGnode{\Lday{\aday}}{\source}$, that is
$\PermutationBijectionCoDomain_{\source,\lday} = \nset{l}{
  \Flow{y^{\mathsf{out}}}{\Lday{\aday}}{\source}{l}=1 }$.  Note that, the two
sets $\PermutationBijectionDomain_{\source,\lday}$ and
$\PermutationBijectionCoDomain_{\source,\lday}$ have same cardinality as
using~\eqref{def:yin}, and~\eqref{def:yout}
and~\eqref{eq:storage_sent_leaving_source} we have 
\begin{align*}
  \cardinal{\PermutationBijectionDomain_{\source,\lday}}=
  \sum\limits_{l \in \MaxArcsSource} \Flow{y^{\mathsf{in}}}{\Lday{\aday}}{\source}{l}
  =
  \sum\limits_{l \in \MaxArcsSource} \Flow{y^{\mathsf{out}}}{\Lday{\aday}}{\source}{l}
  = \cardinal{\PermutationBijectionCoDomain_{\source,\lday}}
  \eqfinp
\end{align*}
The nondecreasing
stock constraint is satisfied if there exists a bijection
$\bijection{\Lday{\aday}}{\source}: \PermutationBijectionDomain_{\source,\lday} \to
\PermutationBijectionCoDomain_{\source,\lday}$ such that
\begin{equation}\label{eq:permutation_constraints}
  \Flow{f^{\mathsf{in}}}{\Lday{\aday}}{\source}{l} \leq
  \Flow{f^{\mathsf{out}}}{\Lday{\aday}}{\source}{\bijection{\Lday{\aday}}{\source}(l)}
 \eqsepv \forall l \in \PermutationBijectionDomain_{\source,\lday} \eqfinp
\end{equation}

\noindent$\bullet$ For all $\source \in \SOURCES$ and $\aday \in \DAYS$, we reduce
symmetry in storage flows by imposing to the vector
$\na{\SFlow{\Lday{\aday}}{\source}{k}}_{k \in\IntegerSet{\kfun(\source)}}$ to be
nonincreasing, that is, when $\kfun(\source)\geq2$, we add for all
$k \in [\kfun(\source)-1]$ the constraint
\begin{equation}\label{eq:forced_arc_source}
  \SFlow{\Lday{\aday}}{\source}{k+1}=1 \implies \SFlow{\Lday{\aday}}{\source}{k}=1 \eqfinp
\end{equation}

\subsubsection{Constraints at nodes in $\SOURCES {\times} \RDAYS$}
There is no refill at nodes $\TEGnode{\Rday{\aday}}{\source}$ as stated in
Assumption~\ref{Asept}, that is, for all $\source \in \SOURCES$ and
$\aday \in \DAYS$
\begin{subequations}\label{eq:no_refill_day}
  \begin{align}
    \HFlow{\Rday{\aday}^-}{\source}{k} &=
                                         \HFlow{\Rday{\aday}}{\source}{k}
                                         \eqsepv \forall k \in\IntegerSet{\kfun(\source)} \eqfinv  \\
    \SFlow{\Rday{\aday}^-}{\source}{k} &=
                                         \SFlow{\Rday{\aday}}{\source}{k}
                                         \eqsepv \forall k \in\IntegerSet{\kfun(\source)} \eqfinp
  \end{align}
\end{subequations}  

\subsection{Cost function}
We aim at minimizing a cost function which is the addition of four
intertemporal costs, namely, a transport cost
\begin{subequations}
\begin{align}
  L^{\tr}(\Storageflow)
  &=c^{\tr} \Big( \sum_{(p_1,p_2,i) \in \LARCS\cup\RARCS}
    t_{(p_1,p_2)}{\times}\SFlow{i}{p_1}{p_2} \Big)
    \eqfinv
    \intertext{a refill cost}
  L^{\re}(r)
  &= \sum_{(\source, \aday) \in \SOURCES{\times}\DAYS}
    c^{\re}_{\source}\nodeV{r}{\Lday{\aday}}{\source}
    \eqfinv
    \intertext{a variable demand dissatisfaction cost}
    L^{\Dvdis}(z)
  &=  c^{\Dvdis} \Big( \sum_{(\destination,\aday) \in \STATIONS{\times}\DAYS}
    \big(\nodeV{\Demand}{\Lday{\aday}}{\destination}
    +\nodeV{\Demand}{\Rday{\aday}}{\destination} 
  \nonumber \\
  &\hspace{2.0cm}
    - (\nodeV{z}{\Lday{\aday}}{\destination}
    + \nodeV{z}{\Rday{\aday}}{\destination})
    \big) \Big)
    \eqfinv
    \intertext{and a fixed demand dissatisfaction cost}
    L^{\Dfdis}(z)
  &=c^{\Dfdis} \Big( \sum_{(\destination,\aday) \in \STATIONS{\times}\DAYS}
     \findi{>0}\big(\nodeV{\Demand}{\Lday{\aday}}{\destination}
    +\nodeV{\Demand}{\Rday{\aday}}{\destination}
  \nonumber \\
  &\hspace{2.0cm}
    - (\nodeV{z}{\Lday{\aday}}{\destination}
    + \nodeV{z}{\Rday{\aday}}{\destination})
    \big) \Big)
    \eqfinv
\end{align}
\end{subequations}
where $t_{(p_1,p_2)}$ is the transport time from $p_1$ to $p_2$, $c^{\tr}$ is
the unit transport cost (\euro/km), $c^{\re}_\source$ is the unit purchase
price of hydrogen (\euro/kg) (refilling price) at source $\source$, $c^{\Dvdis}$ is the variable unit
dissatisfaction cost (\euro/kg) and $c^{\Dfdis}$ is the fixed
dissatisfaction cost (\euro).

\section{Problem formulation and resolution using MILP and heuristics}
\label{se:formulation_solving}
\subsection{Problem formulation}

Gathering all that has been done in Sect.~\ref{sec:modeling}, we 
formulate the following optimization problem as follows
\begin{align}\label{flow_model_problem_formulation}
 \min_{\Storageflow,\Hydrogenflow,r,z, \swap,\sigma} \nonumber
  &L^{\tr}(\Storageflow) + L^{\re}(r)
  \\
  &\hphantom{L^{\tr}} + L^{\Dvdis}(z) +  L^{\Dfdis}(z) 
  \\
  &\text{s.t. } \eqref{eq:flow_arc}{-}\eqref{eq:no_refill_day}
    \nonumber
    \eqfinv
\end{align}
where $\Storageflow$ and $\Hydrogenflow$ are the flow decision variables defined
in Table~\ref{table:flow_decisions} and $r$, $z$, $\swap$ are respectively the
refilling decision at the sources and the emptying and swap decisions at the
destinations as defined in Table~\ref{table:node_decisions}. We denote by
$\mathrm{val}(\mathcal{P})$ the optimal value of
Problem~\eqref{flow_model_problem_formulation}.

\begin{lemma}\label{lemma:linearization}
  The optimization problem~\eqref{flow_model_problem_formulation}, which
  includes nonlinear constraints—specifically \eqref{eq:flow_arc},
  \eqref{eq:emptying_bounded1}, \eqref{eq:swap_destination},
  \eqref{eq:demand_satisfaction2}, \eqref{eq:swap_time},
  \eqref{eq:emtying_bounded2}, \eqref{eq:permutation_constraints},
  \eqref{eq:forced_arc_source}— is equivalent to a MILP, in the sense that the value of the two problems
  coincide and solutions to either problem can be derived from one another.
\end{lemma}

\begin{proof}
  See Appendix~\ref{appendix:linearization} for the proof.
\end{proof}

Solving the MILP reformulation of Problem~\eqref{flow_model_problem_formulation}
is difficult for instances with large number of destinations. This is due to the
quadratic growth of the number of variables needed to linearize the nondecreasing stock
constraints~\eqref{eq:permutation_constraints} with respect to the number of destinations. Therefore, using a solver to
directly tackle Problem~\eqref{flow_model_problem_formulation} may lead to poor
performances. In \S\ref{hympulsion_medium_sized}, we outline the method for
solving Problem~\eqref{flow_model_problem_formulation} up to medium-sized
instances, and in \S\ref{hympulsion_large_sized}, we present a heuristic for
tackling large-size instances.

\subsection{Resolution using MILP and heuristics}

In~\S\ref{hympulsion_medium_sized}, we present the resolution of Problem~\eqref{flow_model_problem_formulation} using a MILP solver, referred to as the \algname{MA} method. In~\S\ref{hympulsion_large_sized}, we introduce a two-step heuristic approach, that we call the \algname{RH} method, to address Problem~\eqref{flow_model_problem_formulation}, and in~\S\ref{se:greedy_heuristic}, we detail a greedy heuristic, called the \algname{GH} method, for solving the same problem. The \algname{GH} method is an improved version of the heuristic employed by the industrial partner for the problem and will serve as a baseline for comparison against \algname{MA} and \algname{RH}.

In all the three presented algorithms a MILP solver, specifically
  Gurobi, is used as a core ingredient. It is executed with a predefined time limit,
  which is the same for all the three algorithms, to prevent excessive
  computational time and ensure practical applicability of the proposed algorithms, and to ensure a fair comparison of the three algorithms. It is important to note that the time limit imposed on the two-step heuristic is only applied to its first step, as the second step executes rapidly, as detailed in the corresponding section. 

\subsubsection{MILP resolution of Problem~\eqref{flow_model_problem_formulation}}
\label{hympulsion_medium_sized}
Problem~\eqref{flow_model_problem_formulation}, which is equivalent to a MILP (see Lemma~\ref{lemma:linearization}), is solved using a MILP solver
(Gurobi), without employing any additional techniques or improvements. For
medium-sized instances, Gurobi demonstrates strong performance and produces
high-quality results. However, its effectiveness diminishes when applied to
large-scale instances. The medium-sized instances considered in this study
include up to 35 destinations. In what
follows, we will refer to the direct use of Gurobi to solve
Problem~\eqref{flow_model_problem_formulation} as \algname{MA}.

\subsubsection{Resolution of Problem~\eqref{flow_model_problem_formulation} using a two-step heuristic}
\label{hympulsion_large_sized}
Using a MILP solver to directly tackle
Problem~\eqref{flow_model_problem_formulation} may lead to poor performances for
large instances due to the nondecreasing stock
constraint~\eqref{eq:permutation_constraints}, which is highly combinatorial.

For that reason, when considering large instances of
Problem~\eqref{flow_model_problem_formulation}, we rely on a heuristic, that we
call \algname{RH}. Before describing this heuristic, we consider two additional
optimization problems. The first one is the same optimization problem as
Problem~\eqref{flow_model_problem_formulation}, but without the permutation
constraint~\eqref{eq:permutation_constraints}, that is,

\begin{align}\label{flow_model_problem_formulation_without_permutation} 
  \min_{\Storageflow,\Hydrogenflow,r,z, \swap,\sigma} \nonumber
  &L^{\tr}(\Storageflow) + L^{\re}(r)
  \\
  &\hphantom{L^{\tr}} + L^{\Dvdis}(z) +  L^{\Dfdis}(z) 
  \\
  &\text{s.t. } \eqref{eq:flow_arc}{-}\eqref{eq:refill_bounded},\eqref{eq:forced_arc_source},\eqref{eq:no_refill_day} 
    \nonumber
    \eqfinp
\end{align}
We denote by $\mathrm{val}(\widetilde{\mathcal{P}})$ the optimal value of
Problem~\eqref{flow_model_problem_formulation_without_permutation}. Since the
feasible set of Problem~\eqref{flow_model_problem_formulation} in included in
feasible set of Problem~\eqref{flow_model_problem_formulation_without_permutation}, we have that
\begin{equation}
  \mathrm{val}(\mathcal{\widetilde{P}}) \leq \mathrm{val}(\mathcal{P}) \eqfinp
\end{equation}
Consequently, a lower bound of
Problem~\eqref{flow_model_problem_formulation_without_permutation} is also a
lower bound for Problem~\eqref{flow_model_problem_formulation}.

The second optimization problem is an equivalent formulation of
Problem~\eqref{flow_model_problem_formulation}, which is given by
\begin{align}\label{reorganization_flow_problem}
\min_{\Storageflow}  &\: L^{\tr}(\Storageflow) + \phi(\Storageflow)
  \\
  &\text{s.t. }
    \eqref{eq:always_storage_destination},
    \eqref{eq:at_most_one_storage_destination},
    \eqref{eq:storage_in_storage_out_destination},
    \eqref{eq:limit_storage_source},
    \eqref{eq:storage_sent_leaving_source},
    \eqref{eq:forced_arc_source}
    \nonumber \eqfinv
\end{align}
where the function $\phi(\Storageflow)$ is the optimal value of 
\begin{align}\label{hympulsion_flow_subproblem}
\min_{\Hydrogenflow,r,z, \swap, \sigma}
    &L^{\re}(r) + L^{\Dvdis}(z) +  L^{\Dfdis}(z)
  \\
  &\text{s.t. }
    \eqref{eq:flow_arc},
    \eqref{eq:demand_satisfaction1}{-}%
    \eqref{eq:emptying_bounded1},
    \eqref{eq:swap_destination}{-}%
    \eqref{eq:emtying_bounded2},
  \nonumber\\
  &\hphantom{\text{s.t. }}
    \eqref{eq:flow_conservation_source}{-}
    \eqref{eq:permutation_constraints},
    \eqref{eq:no_refill_day}
    \nonumber \eqfinp
\end{align}
We also denote by $\gamma(\Storageflow)$ the optimal solution of
Problem~\eqref{hympulsion_flow_subproblem}. Given a feasible storage flow
$\Storageflow$, that is, a storage flow $\Storageflow$ satisfying the
constraints of Problem~\eqref{reorganization_flow_problem}, computing the value
of $\phi(\Storageflow)$ amounts to solve
Problem~\eqref{hympulsion_flow_subproblem} on the subgraph of $\GRAPH^{\TEG}$
generated by active arcs of $\Storageflow$, which is a MILP with much fewer
integer variables than in Problem~\eqref{flow_model_problem_formulation} as we
only consider the nondecreasing stock
constraint~\eqref{eq:permutation_constraints} on this subgraph. As a
consequence, computing $\phi(\Storageflow)$ is tractable even for large-size
instances.

Now, we detail the \algname{RH} method to solve large-sized instances of
Problem~\eqref{reorganization_flow_problem}. First, note that solving
Problem~\eqref{flow_model_problem_formulation_without_permutation} is easier
then Problem~\eqref{flow_model_problem_formulation} since it does not include
the nondecreasing stock constraint~\eqref{eq:permutation_constraints}. Second,
for a feasible solution $\Storageflow$ of
Problem~\eqref{flow_model_problem_formulation_without_permutation}, that is, for
a solution $\Storageflow$ that satisfies
Constraints~\eqref{eq:always_storage_destination},
\eqref{eq:at_most_one_storage_destination},
\eqref{eq:storage_in_storage_out_destination}, \eqref{eq:limit_storage_source},
\eqref{eq:storage_sent_leaving_source} and \eqref{eq:forced_arc_source}, solving
Problem~\eqref{hympulsion_flow_subproblem} is easy as previously mentioned, and
the whole solution $(\Storageflow,\gamma(\Storageflow))$ is feasible for the original
Problem~\eqref{flow_model_problem_formulation} as $\Storageflow$ is feasible for
Problem~\eqref{reorganization_flow_problem} and $\gamma(\Storageflow)$ is feasible
for Problem~\eqref{hympulsion_flow_subproblem}. Therefore, the heuristic
approach first solves
Problem~\eqref{flow_model_problem_formulation_without_permutation} with a
computation time limit to obtain a feasible storage flow~$\Storageflow$, and
then, solves Problem~\eqref{hympulsion_flow_subproblem} to obtain a feasible
solution for Problem~\eqref{flow_model_problem_formulation}. For what
follows, we refer to the lower bound obtained by \algname{RH} as the lower bound
obtained when solving
Problem~\eqref{flow_model_problem_formulation_without_permutation} in the first
step of the heuristic. A short pseudocode for \algname{RH} is given in
Algorithm~\ref{flow_hympulsion_relaxed_heuristic}

Surprisingly, \algname{RH} yields better results compared to using \algname{MA}
for large instances. It also produces tighter lower bounds, which facilitates verifying the near-optimality of the solutions found.
 
\begin{algorithm*}
  \caption{Two-step heuristic (\algname{RH}) for
    Problem~\eqref{flow_model_problem_formulation}}
  \label{flow_hympulsion_relaxed_heuristic}
  \begin{algorithmic}[1]
    \State Solve Problem~\eqref{flow_model_problem_formulation_without_permutation} with a time
    limit to obtain $\Storageflow$
    \State Solve Problem~\eqref{hympulsion_flow_subproblem} to obtain $\gamma(\Storageflow)$
    \State \Return ($\Storageflow, \gamma(\Storageflow)$)
  \end{algorithmic}
\end{algorithm*}

\begin{remark}
  One might consider applying Benders decomposition to
Problem~\eqref{reorganization_flow_problem} by approximating the function $\phi$
using a set of cuts. However, as demonstrated in the numerical results section,
the \algname{RH} method is enough to
obtain high-quality solutions to Problem~\eqref{flow_model_problem_formulation}.
Moreover, the nonconvexity of $\phi$ with respect to
$\Storageflow$, as Problem~\eqref{hympulsion_flow_subproblem} involves binary
variables, makes the application of Benders decomposition more
challenging. Although some techniques have been proposed to handle such
nonconvexity in the context of Benders
decomposition~\cite{BendersDecompoNonConvex}, there is no guarantee that they
will be effective for Problem~\eqref{reorganization_flow_problem}.
\end{remark}
\subsubsection{A greedy heuristic for Problem~\eqref{flow_model_problem_formulation}}\label{se:greedy_heuristic}

We introduce a third approach --- which is an improved version of the heuristic
employed by the industrial partner for the problem ---, that we will call
\algname{GH}, to serve as a baseline for comparison against the methods
described in~\S\ref{hympulsion_medium_sized} and
in~\S\ref{hympulsion_large_sized}. While \algname{GH} is suboptimal, it provides
solutions that are fast to compute.

In Algorithm~\ref{flow_hympulsion_heuristic}, we present the pseudocode of
\algname{GH}, which is composed of two steps. First, it finds a feasible storage
flow $\Storageflow$ to Problem~\eqref{reorganization_flow_problem} and, second,
solves Problem~\eqref{hympulsion_flow_subproblem} (which is a tractable MILP as
already explained) to obtain a feasible solution
$(\Storageflow,\gamma(\Storageflow))$ to the global
Problem~\eqref{flow_model_problem_formulation}.

During the first step, it is assumed that the sources have sufficient refilling
capacity to fully replenish the storages.  More precisely, for a given day
$\aday$, if a storage is sent from a source~$\source$ to a destination $d$, that
is, if $\SFlow{\Lday{\aday}}{\source}{\destination}$ is set to the value $1$,
then the storage is sent full, that is,
$\HFlow{\Lday{\aday}}{\source}{\destination}$ is set to the value
$\overline{\Stock}$.  This assumption may violate
constraints~\eqref{eq:flow_conservation_source} and~\eqref{eq:refill_bounded}
but at the end of the first step, the intermediate hydrogen flow $\Hydrogenflow$
is not kept and it is recomputed when solving
Problem~\eqref{hympulsion_flow_subproblem}.

Now, we detail the first step. We sequentially iterate on the set $\DAYS$ of
days, to build an admissible storage flow $\Storageflow$ together with a
compatible hydrogen flow $\Hydrogenflow$ as follows.

\begin{enumerate}
\item At the start of day $\aday \in \DAYS$, having already computed
  $\Storageflow$ and $\Hydrogenflow$ for the previous days, we compute a subset
  of destinations
  $\STATIONSCRITIC = \nset{\destination \in
    \STATIONS}{\HFlow{\Lday{\aday}^-}{\destination}{\destination} \leq
    \StockCritic}$, referred to as the \emph{critical destinations}, whose
  hydrogen inflow $\HFlow{\Lday{\aday}^-}{\destination}{\destination}$ have
  fallen below a predefined threshold $\StockCritic$. The destinations belonging to the subset
  $\STATIONSCRITIC$ are selected for replenishment and are sorted in ascending
  order with respect to their hydrogen inflow.

\item Then, we iterate on the ordered critical destinations as follows. The
  closest source to the current destination with available storages is selected
  and one of its storages is sent (full, according to
  previous assumption made) to the current destination
  and removed from the available storages. The storage that was already present at the
  destination is returned to the same source and
  will be available the next day.

\item The iteration on the ordered critical destinations for the current day
  stops when all the ordered critical destinations have been served or if there
  is no more available storages at the sources.

\item Finally, the hydrogen flow~$\Hydrogenflow$ for day $\aday$ is updated
  using Equations~\eqref{eq:demand_satisfaction1}
  and~\eqref{eq:demand_satisfaction2}.
\end{enumerate}

\begin{algorithm*}
  \caption{Greedy heuristic (\algname{GH}) for Problem~\eqref{flow_model_problem_formulation}}
  \label{flow_hympulsion_heuristic}
  \begin{algorithmic}[1] \Require parameter $\StockCritic$
    \Comment{the critical stock threshold}
    \For{$\aday = 1,..,J$}

   \For{$\source \in \SOURCES$}
  
  \State $V[\source] \gets \sum\limits_{\destination \in \STATIONS}
  \SFlow{\Rday{\aday-1}}{\destination}{\source} + \sum\limits_{k
    \in\IntegerSet{\kfun(\source)}} \SFlow{\Rday{\aday-1}}{\source}{k}$
  \Comment{\hspace{-0.1cm}number of available storages at $\source$ the day $j$}

  \EndFor
  
  \State $\STATIONSCRITIC \gets \nset{\destination \in
    \STATIONS}{\HFlow{\Lday{\aday}^-}{\destination}{\destination} \leq
    \StockCritic}$ \label{greedy_algo:5} \Comment{Initialize the critical destinations}

  \State Sort $\STATIONSCRITIC$ in ascending order of
  $\HFlow{\Lday{\aday}^-}{\destination}{\destination}$\label{greedy_algo:6}

  \For{$\destination \in \STATIONSCRITIC$}

  \State$s' \gets \argmin\limits_{\substack{\source \in \SOURCES \\
      V[\source] \geq 1}} g(\source,\destination)$
  \hspace{-1.5cm}\label{greedy_algo:8}\Comment{Select the closest source with available storages}

  \If{$s' = \emptyset$}\label{greedy_algo:9} break\EndIf
  
  \State $\SFlow{\Lday{\aday}}{\source'}{\destination} \gets$ 1
  \label{greedy_algo:11}\Comment{Send the full storage from $\source'$ to $\destination$}
  
  \State $V[\source']\gets V[\source']-1$
  \label{greedy_algo:12}\Comment{Decrease the number of available storages at $\source'$}

  \State $\SFlow{\Rday{\aday}}{\destination}{\source'} \gets 1$
  \label{greedy_algo:13}\Comment{Return the storage from $\destination$ back to the same source
    $\source'$}

  \State Update $\Hydrogenflow$ using
  \label{greedy_algo:14}Equations~\eqref{eq:demand_satisfaction1} and \eqref{eq:demand_satisfaction2}
  
  \EndFor

  \EndFor

  \State Compute $\gamma(\Storageflow)$ \Comment{$\gamma(\Storageflow)$ is
    the optimal solution of Problem~\eqref{hympulsion_flow_subproblem}}
  
  \State \Return ($\Storageflow, \gamma(\Storageflow)$)
\end{algorithmic}
\end{algorithm*}

\section{Computational experiments}\label{se:results}
We aim to analyze the quality of the solutions for
Problem~\eqref{flow_model_problem_formulation} provided by the three methods
previously introduced, namely \algname{MA} (in~\S\ref{hympulsion_medium_sized}), \algname{RH} (in~\S\ref{hympulsion_large_sized}) and \algname{GH} (in~\S\ref{se:greedy_heuristic}).

\subsection{Experimental Setting}

\subsubsection{Instances}
We describe now the set of instances used for numerical tests.
We vary the number of sources, destinations, and storages -- to reflect
both what is found in the current infrastructure and
what will be found in anticipated possible expansions --
together with their characteristics and transport characteristics.

Each instance is characterized by its number of sources, number of destinations,
number of storages, demand profile magnitude and dissatisfaction cost magnitude
for a total of $192$ instances. The difficulty of an instance is mainly driven by its number of locations and more precisely its number of destinations, as it will be discussed later. We describe now more precisely how the different
parameters vary in the instances.

\begin{itemize}
\item Characteristics of the sources 
  \begin{itemize}

  \item The number of sources ranges from 1 to 7, that is $1\leq \cardinal{\SOURCES} \leq 7$.  
  \item The refilling capacity and price at each source are given in
    Table~\ref{tab:source_capacity_price}.
\begin{table}[h!]
\centering
\begin{tabular}{|c|r|r|}
  \hline
  Source & Refilling capacity & Refilling price \\
         & $\overline{r}_{\source}$ & $c^{\re}_{\source}$ \\
  \hline
  $\source_1$ & 1300 & \numprint{9.0} \\ \hline
  $\source_2$ & 1500 & \numprint{8.0} \\ \hline
  $\source_3$ & 1700 & \numprint{8.3} \\ \hline
  $\source_4$ & 1000 & \numprint{8.0} \\ \hline
  $\source_5$ & 1000 & \numprint{8.0} \\ \hline
  $\source_6$ & 800 & \numprint{10.0} \\ \hline
  $\source_7$ & 500 & \numprint{7.0} \\ \hline
\end{tabular}
\caption{Refilling capacity and price of the sources, where $s_l$ is the $l$-th
  source of $\SOURCES$.}
\label{tab:source_capacity_price}
\end{table}

\item The maximum number $\kfun(\source)$ of storages that can be present at
  source $\source$, is equal to 4, for all $\source \in \SOURCES$.
  \end{itemize}
  \item Characteristics of the destinations
    \begin{itemize}
    \item The ratio ${\cardinal{\STATIONS}}/{\cardinal{\SOURCES}}$ ranges
      from \numprint{4.33} to \numprint{8.5}.
    \item Demand follows a three-peaked daily pattern, with peaks at 8:00,
      14:00, and 17:00. On Saturdays and Sundays, demand reduces to 50\% and
      25\%, respectively, of the weekday level.
    \item Each instance is characterized by a demand profile magnitude and a
      demand dissatisfaction cost magnitude. There are two types of the demand
      profile magnitude. For first one, the average weekday demand at each
      destination is \numprint[kg]{85} per day. For the second one, the average
      weekday demand at each destination is \numprint[kg]{130} par day. In
      Figure~\ref{fig:destination_demand_exemple}, we show an example of a
      demand profile for a weekday at a destination with an average daily demand
      of \numprint[kg]{85}. There are also two types of the demand
      dissatisfaction cost magnitude. For the first one, the variable demand
      dissatisfaction cost~$c^{\Dvdis}$ is equal to \numprint{12}\euro/kg, and
      the fixed demand dissatisfaction cost~$c^{\Dfdis}$ is equal to
      \numprint{1500}\euro. For the second one, the variable demand
      dissatisfaction cost, $c^{\Dvdis}$, is equal to \numprint{14}\euro/kg and
      the fixed demand dissatisfaction cost, $c^{\Dfdis}$, is equal to
      \numprint{2500}\euro.
      \item The swap time at the destinations is equal to \numprint[hour]{1}.
      \end{itemize}
    \item Characteristics of the storages
      \begin{itemize}
      \item The maximal stock capacity $\StockMax$ of the storages is equal to \numprint[kg]{300}.
      \item The ratio ${\cardinal{\STORAGES}}/{\cardinal{\STATIONS}}$ ranges from \numprint{1.26} to \numprint{1.5}.
      \item The stock capacity at the beginning of the time span of the storages
        that are located at the source is equal to \numprint[kg]{200}.
      \item The stock capacity at the beginning of the time span of the storages
        that are located at the destinations is equal to \numprint[kg]{200}.
      \end{itemize}
      
    \item Characteristics of the transport
      \begin{itemize}
      \item The distances between the sources and the destinations are randomly
        generated within a range of \numprint[km]{1} to \numprint[km]{123},
        based on the values of the case study. Specifically, for all
        $(p_1, p_2) \in (\SOURCES \times \STATIONS) \cup (\STATIONS \times \SOURCES)$, we have
        that \(1 \leq t_{(p_1, p_2)} = t_{(p_2, p_1)} \leq 123\).
      \item The transport cost $c^{\tr}$ is equal to \numprint{2.25}\euro/km.
      \item The trucks arrive at $h_0=8$ at the sources to transport the storages.
        \end{itemize}
\end{itemize}

\begin{figure}[htpp]
  \centering
  \mbox{\includegraphics[width=0.5\textwidth]{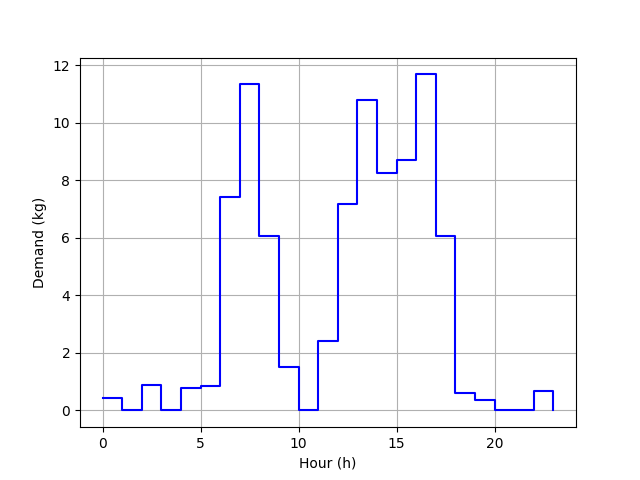}}
  \caption{Example of a demand profile for a weekday at a destination with an average daily demand of \numprint[kg]{85}}
  \label{fig:destination_demand_exemple}
\end{figure}

\subsubsection{Implementation}
The MILP model is implemented in Julia 1.9.2. using JuMP~\cite{JuMP} as the
modeler and Gurobi 11.0~\cite{gurobi} as the MILP solver. All computations were
performed on a Linux system equipped with 4-processor Intel Xeon E5{-}2667,
3.30GHz, with 192 GB of RAM. For each instance and for each algorithm, the
maximum execution time is set to \numprint[min]{20} (\algname{GH} method
finishes way before this time limit)

\subsection{Performance analysis}
\subsubsection{Total cost comparison}
The most straightforward metric for evaluating our algorithms is the cost
function of Problem~\eqref{flow_model_problem_formulation} that is being
minimized. In Figure~\ref{fig:boxplot_total_cost}, we show the cost obtained by
each algorithm over the whole instances summarized as boxplots. The red line
(resp. the purple triangle) indicates the median (resp. the mean) over the
instances, the ends of the boxes the first and third quartiles. The whiskers extend to either 1.5 times the interquartile range (the difference between the third and first quartiles) or the last data point within this range, whichever is closer, and any points beyond are considered outliers.

\begin{figure}[htpp]
  \centering
  \mbox{\includegraphics[width=0.5\textwidth]{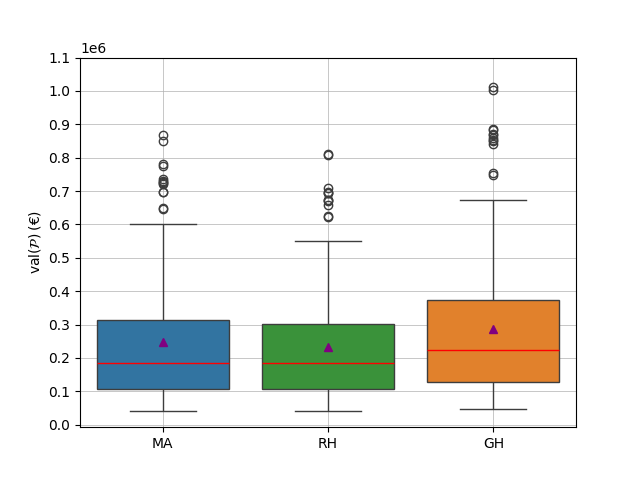}}
  \caption{Cost obtained by each algorithm over the whole instances summarized as boxplots, where the lower the mean (represented by the purple triangle) the better}
  \label{fig:boxplot_total_cost}
\end{figure}

\begin{result}
  The best solution given by the \algname{MA} and \algname{RH} methods is always
  better than the solution returned by the \algname{GH} method for all the
  instances. Moreover, the median (resp. the mean) obtained by the \algname{MA}
  and \algname{RH} methods outperform the median (resp. the mean) obtained by
  the \algname{GH} method by \numprint[\%]{17.7} and \numprint[\%]{17.8}
  (resp. \numprint[\%]{14.1} and \numprint[\%]{18.9}).
\end{result}


\subsubsection{Total cost comparison per number of destinations}
As shown in Figure~\ref{fig:boxplot_total_cost}, there is a slight difference
between the cost distribution returned by \algname{MA} and by \algname{RH} and,
as previously mentioned, we expect \algname{RH} to perform better on large
instances.  To investigate this, we assign to each instance an indicator, that we call
\emph{Q\_destination}, based on its number of destinations. This indicator takes four distinct values, ranging from \( Q_1 \) to \( Q_4 \), following the classification criteria outlined in Table~\ref{tab:indicator_quartile}. For example, an instance is assigned the value \( Q_2 \) if its number of destinations falls within the range of 19 to 27. Note that the threshold values in Table~\ref{tab:indicator_quartile} are selected to ensure an equal distribution of instances across all the indicator values, resulting in exactly 48 instances per category.

\begin{table}[h]
    \centering
    \begin{tabular}{|c|c|}
        \hline
        Q\_destination & $\cardinal{\STATIONS}$ \\
        \hline
        $Q_1$ & [10,18] \\
        \hline
        $Q_2$ & [19,27] \\
        \hline
        $Q_3$ & [28,35] \\
        \hline
        $Q_4$ & [36,48] \\
        \hline
    \end{tabular}
    \caption{Threshold values defining the classification of instances according to their number of destinations.}

    \label{tab:indicator_quartile}
  \end{table}
  


In Figure~\ref{fig:boxplot_cost_quartile_destination}, we show the
cost obtained by each algorithm over the whole instances summarized as boxplots,
with an additional partitioning using the indicator Q\_destination.

\begin{figure}[htpp]
\centering
\mbox{\includegraphics[width=0.5\textwidth]{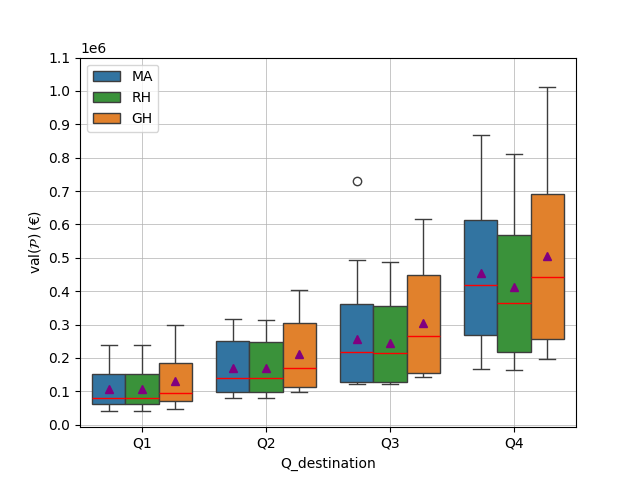}}
  \caption{Cost obtained by each algorithm over the whole instances summarized as boxplots, along with an additional partitioning using the indicator Q\_destination, where a higher value of this indicator corresponds to a higher number of destinations}
  \label{fig:boxplot_cost_quartile_destination}
\end{figure}

\begin{result}
  The \algname{MA} and \algname{RH} methods show comparable performance for the first three values of the indicator Q\_destination, that is $Q_1$, $Q_2$ and $Q_3$. However, for the value $Q_4$, the median
  (resp. the mean) obtained by the \algname{RH} algorithm outperforms the median
  (resp. the mean) obtained by the \algname{MA} algorithm by \numprint[\%]{12.7}
  (resp. \numprint[\%]{9}).
\end{result}

\subsubsection{Relative gap comparison per number of destinations}
For a given algorithm, the relative gap of an instance is defined as the difference between an upper bound given by the value of the best solution and a lower bound returned by the algorithm, divided by the same upper bound.

We show in Figure~\ref{fig:boxplot_gap_quartile_destination} the relative gap
obtained by \algname{MA} and \algname{RH} algorithms over the whole instances
summarized as boxplots, with an additional partitioning using the indicator
Q\_destination.

\begin{figure}[htpp]
  \centering
  \mbox{\includegraphics[width=0.5\textwidth]{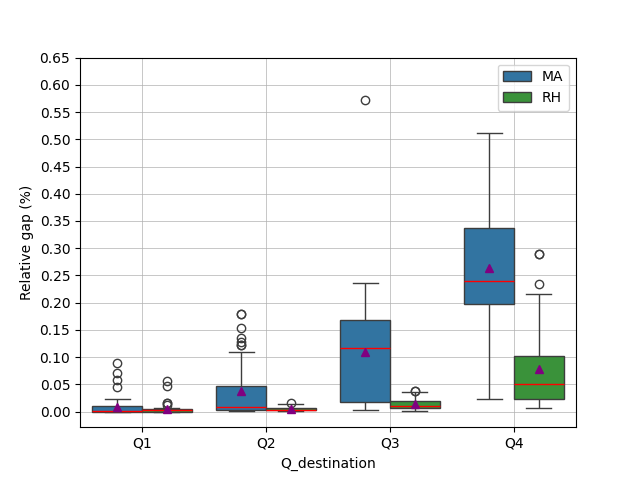}}
  \caption{Relative gap obtained by \algname{MA} and \algname{RH} algorithms over the whole instances summarized as boxplots, along with an additional partitioning using the indicator Q\_destination given in Table~\ref{tab:indicator_quartile}}
  \label{fig:boxplot_gap_quartile_destination}
\end{figure}

\begin{result}
  The relative gap of the \algname{MA} algorithm increases significantly as the
  number of destinations grows, reducing its ability to effectively determine
  the quality of the solution. This relative gap for $Q_4$ is equal
  to \numprint[\%]{23.9} in median and \numprint[\%]{26.4} in average. In
  contrast, the lower bounds found using the \algname{RH} algorithm are tighter, with a relative gap for $Q_4$ equals to \numprint[\%]{5} in median and \numprint[\%]{7.7} in average.
\end{result}

\subsubsection{Relative gap comparison according to the demand satisfaction}
For each instance, we introduce an indicator called \emph{S\_demand}, which
takes the value ``yes" if all the demands at the destinations are satisfied, and
``no" otherwise. Then, we present in
Figure~\ref{fig:boxplot_gap_is_demand_satisfied} the relative gap obtained by
\algname{MA} and \algname{RH} algorithms over the whole instances summarized as
boxplots, with an additional partitioning using the indicator S\_demand.

\begin{figure}[htpp]
  \centering
  \mbox{\includegraphics[width=0.5\textwidth]{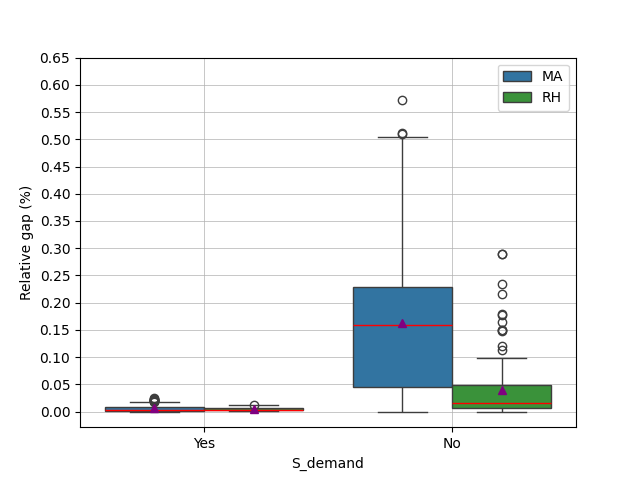}}
  \caption{Relative gap obtained by \algname{MA} and \algname{RH} algorithms over the whole instances summarized as boxplots, along with an additional partitioning using the indicator S\_demand}
  \label{fig:boxplot_gap_is_demand_satisfied}
\end{figure}

\begin{result}
  When the best solution found fully satisfies the demand, both algorithms find tight lower bound
  that guarantee the near-optimality of the solution. However, when the demand
  is not fully satisfied, \algname{RH}'s ability to determine whether
  the solution is optimal decreases slightly, while the relative gap produced by
  \algname{MA} increases significantly. In other words, if the \algname{MA}'s best solution
  fails to satisfy all the demands, near-optimality of the solution is not
  guaranteed.
\end{result}

\section{From flow to storage transport planning}
\label{sec:flow_to_planning}

In this section, we show in Proposition~\ref{prop:tp} how to derive a transport planning for all
storages $\storage \in \STORAGES$ given a solution of
Problem~\eqref{flow_model_problem_formulation}. In order to prove Proposition~\ref{prop:tp}, we start by
notations and preliminary result.

\begin{lemma}
  \label{le:sigma_bar}
  Given an admissible solution of
  Problem~\eqref{flow_model_problem_formulation}, for all $s\in \SOURCES$
  and all $\lday \in \LDAYS$, we can extend the bijection
  $\bijection{\Lday{\aday}}{\source}:
  \PermutationBijectionDomain_{\source,\lday} \to
  \PermutationBijectionCoDomain_{\source,\lday}$ to a bijection
  $\obijection{\Lday{\aday}}{\source}: \MaxArcsSource \to \MaxArcsSource$, which
  satisfies
  \begin{equation}\label{eq:permutation_constraints_b}
    \Flow{y^{\mathsf{in}}}{\Lday{\aday}}{\source}{l} 
    = \Flow{y^{\mathsf{out}}}{\Lday{\aday}}{\source}{\obijection{\Lday{\aday}}{\source}(l)}
    \eqsepv \forall l \in \MaxArcsSource \eqfinp
  \end{equation}
\end{lemma}
\begin{proof} We fix $\source \in \SOURCES$ and $\lday \in \LDAYS$.
  As $\cardinal{\PermutationBijectionDomain_{\source,\lday}} =
  \cardinal{\PermutationBijectionCoDomain_{\source,\lday}}$, the set of bijections
  from $\MaxArcsSource\backslash \PermutationBijectionDomain_{\source,\lday}$ to
  $\MaxArcsSource \backslash  \PermutationBijectionCoDomain_{\source,\lday}$ is nonempty
  and we pick one such bijection $\tbijection{\Lday{\aday}}{\source}:  
  \MaxArcsSource\backslash \PermutationBijectionDomain_{\source,\lday}\to
  \MaxArcsSource \backslash  \PermutationBijectionCoDomain_{\source,\lday}$. 
  Defining~$\obijection{\Lday{\aday}}{\source}$ by
  \begin{equation}
    \obijection{\Lday{\aday}}{\source}(l) =
    \begin{cases}
      \bijection{\Lday{\aday}}{\source}(l)
      & \text{if }
        l \in \PermutationBijectionDomain_{\source,\lday}\eqfinv \\
      \tbijection{\Lday{\aday}}{\source}(l)
      & \text{if }
        l \in \MaxArcsSource\backslash \cardinal{\PermutationBijectionDomain_{\source,\lday}}
        \eqfinv
    \end{cases}
  \end{equation}
  it is straightforward to see that~\eqref{eq:permutation_constraints_b} is
  satisfied as storage flows of value 1 are propagated by $\bijection{\Lday{\aday}}{\source}$
  and storage flows of value 0 are propagated by $\tbijection{\Lday{\aday}}{\source}$.
\end{proof}

Now, we introduce some notations.
For all $\Lday{\aday} \in \LDAYS$, we consider the  mappings
$\theta_{\Lday{\aday}}$ defined as follows
\begin{subequations}
  \label{eq:thetaLday}
  \begin{align} 
    \theta_{\Lday{\aday}}: \ARCS^{\TEG}_{\Lday{\aday}^-}
    & \to \ARCS^{\TEG}_{\Lday{\aday}} \nonumber
    \\
    (\source{:}k, \Lday{\aday}^-)
    & \mapsto
      \begin{cases}
        (\source, \destination_{l'}, \Lday{\aday})
        & \text{if }
          l' \le \cardinal{\STATIONS}
        \\
        (\source,l'{-}\cardinal{\STATIONS}, \Lday{\aday})
        & \text{otherwise}
        \\
        \text{with }
        l' =\obijection{\Lday{\aday}}{\source}&\hspace{-0.35cm}(k+\cardinal{\STATIONS})
      \end{cases}\label{eq:theta1}
    \\
    (d_l, \source, \Lday{\aday}^-)
    &\mapsto  \begin{cases}
      (\source, \destination_{l'}, \Lday{\aday})
      & \text{if }
        l' \leq \cardinal{\STATIONS}
      \\
      (\source,l'-\cardinal{\STATIONS}, \Lday{\aday})
      & \text{otherwise}
      \\
      \text{with } l'=\obijection{\Lday{\aday}}{\source}(l) 
    \end{cases}\label{eq:theta2}
  \\
    (d, d , \Lday{\aday}^-)
    & \mapsto (d,d, \Lday{\aday})\label{eq:theta3}
      \eqfinv
  \end{align}
\end{subequations}
where, for all $\Lday{\aday} \in \EDAYS$, the subset of arcs
$\ARCS^{\TEG}_{\aday}$ is defined by
$\ARCS^{\TEG}_{\aday} = \nset{\arc \in \ARCS^{\TEG}}{\dayarc{a}=j}$,
where $\textsf{d}$ was defined in \S\ref{sbsec:def_arcs}.

\begin{lemma}
  \label{lem:theta_under}
  For all $\Lday{\aday} \in \LDAYS$,
  the mapping $\theta_{\Lday{\aday}}$ defined in Equation~\eqref{eq:thetaLday}
  is well-defined, satisfies
  $\tailarc{\theta_{\Lday{\aday}}(\arc)}= \headarc{\arc}$ and is bijective.
  Moreover, it preserves the value of the storage flow as we have
  $\Storageflow_{\arc}= \Storageflow_{\theta_{\Lday{\aday}}(\arc)}$ for all
  $\arc \in \ARCS^{\TEG}_{\Lday{\aday}^-}$.
\end{lemma}
\begin{proof}
  The mapping \( \theta_{\Lday{\aday}} \) is well-defined, as by construction
  it maps each element of
  \( \ARCS^{\TEG}_{\Lday{\aday}^-} \) to a corresponding element in
  \( \ARCS^{\TEG}_{\Lday{\aday}} \). From the definition in Equation~\eqref{eq:thetaLday} of the mapping
  \( \theta_{\Lday{\aday}} \), it is straightforward to check that
  \( \tailarc{\theta_{\Lday{\aday}}(\arc)} = \headarc{\arc} \) for any
  \( \arc \in \ARCS^{\TEG}_{\Lday{\aday}^-} \). Furthermore,
  \( \theta_{\Lday{\aday}} \) is bijective, as the underlying mapping
  \( \obijection{\Lday{\aday}}{\source} \) built in Lemma~\ref{le:sigma_bar}
  is itself bijective.
  
  Now, we check that the storage flows are preserved.  First, if
  $\arc=(\source{:}k,\Lday{\aday}^-)$ or
  $\arc=(\destination_l,\source,\Lday{\aday}^-)$, we have that
  \( \Storageflow_{\arc} = \Storageflow_{\theta_{\Lday{\aday}}(\arc)} \) using
  Equation~\eqref{eq:permutation_constraints_b}. Second, if
  $\arc=(\destination,\destination,\Lday{\aday}^-)$, the equality
  \( \Storageflow_{\arc} = \Storageflow_{\theta_{\Lday{\aday}}(\arc)} \) is satisfied
  using Equation~\eqref{eq:always_storage_destination} which gives that
  \( \Storageflow_{\arc}\) and \(\Storageflow_{\theta_{\Lday{\aday}}(\arc)} \) are
  both equal to one which ends the proof.
\end{proof}

For $\Rday{\aday} \in \RDAYS$, we consider the mappings $\theta_{\Rday{\aday}}$ defined as follows

\begin{subequations}
  \label{eq:thetaRday}
  \begin{align}
    \theta_{\Rday{\aday}}: \ARCS^{\TEG}_{\Rday{\aday}^-}
    & \to \ARCS^{\TEG}_{\Rday{\aday}}
      \nonumber \\
    (\source{:}k, \Rday{\aday}^-)
    &\mapsto (\source{:}k, \Rday{\aday})
      \label{eq:thetaRday1}   \\
    (\source, d, \Rday{\aday}^-)
    &\mapsto 
      \begin{cases}
        (d, d, \Rday{\aday}) & \text{if }
                               \Storageflow_{\source,d, \Rday{\aday}^-} =1
        \\
        (d,\hbijection{\Rday{\aday}}{\destination}(\source), \Rday{\aday})
                             & \text{otherwise }
      \end{cases}
      \label{eq:thetaRday2}  \\
    (d, d , \Rday{\aday}^-) & \mapsto
                              \begin{cases}
                                (\destination, \source, \Rday{\aday})
                                & \text{if }
                                  \exists \source \text{ s.t. }
                                  \Storageflow_{\destination,\source, \Rday{\aday}} =1 
                                \\
                                (d,d, \Rday{\aday})
                                & \text{otherwise}\eqfinv
                              \end{cases}
                              \label{eq:thetaRday3}
  \end{align}
\end{subequations}
where $\hbijection{\Rday{\aday}}{\destination}$ is any bijection from the subset
$\nset{\source \in
  \SOURCES}{\Storageflow_{\source,\destination,\Rday{\aday}^-}=0}$ to the subset
$\nset{\source \in \SOURCES}{\Storageflow_{\destination,\source,\Rday{\aday}}=0}$,
which is guaranteed to exist as the two subsets
$\nset{\source \in
  \SOURCES}{\Storageflow_{\source,\destination,\Rday{\aday}^-}=0}$ and
$\nset{\source \in \SOURCES}{\Storageflow_{\destination,\source,\Rday{\aday}}=0}$
have the same cardinality using
Equation~\eqref{eq:storage_in_storage_out_destination}.

\begin{lemma}\label{lem:theta_above}
  For $\Rday{\aday} \in \RDAYS$, the mapping $\theta_{\Rday{\aday}}$ defined in
  Equation~\eqref{eq:thetaRday} is well-defined, bijective and satisfies
  $\tailarc{\theta_{\Rday{\aday}}(\arc)}= \headarc{\arc}$.  Moreover, it preserves
  the value of the storage flow as we have
  $\Storageflow_{\arc}= \Storageflow_{\theta_{\Rday{\aday}}(\arc)}$ for all
  $\arc \in \ARCS^{\TEG}_{\Rday{\aday}^-}$.
\end{lemma}

\begin{proof}
  The mapping \( \theta_{\Rday{\aday}} \) is well-defined, as by construction it maps
  each element of \( \ARCS^{\TEG}_{\Rday{\aday}^-} \) to a corresponding element
  in \( \ARCS^{\TEG}_{\Rday{\aday}} \).
  From the definition in Equation~\eqref{eq:thetaRday} of the mapping
  \( \theta_{\Rday{\aday}} \), it is straightforward to check that
  \( \tailarc{\theta_{\Rday{\aday}}(\arc)} = \headarc{\arc} \) for any
  \( \arc \in \ARCS^{\TEG}_{\Rday{\aday}^-} \).
  Now, we turn to the proof that the mapping \( \theta_{\Lday{\aday}} \) is bijective.
  Since both the domain and codomain of $\theta_{\Rday{\aday}}$ have the same cardinality,
  that is \( \cardinal{\ARCS^{\TEG}_{\Rday{\aday}^-}} = \cardinal{\ARCS^{\TEG}_{\Rday{\aday}}} \)
  it is enough to prove that \( \theta_{\Rday{\aday}} \) is surjective to obtain that it is bijective.
  To prove that \( \theta_{\Rday{\aday}} \) is surjective we proceed by cases in the arc set
  $\ARCS^{\TEG}_{\Rday{\aday}}$.
  
  $\bullet$ First, for an arc of the form \( (\source{:}k, \Rday{\aday}) \), it follows directly from
  Equation~\eqref{eq:thetaRday1} that
  \( \theta_{\Rday{\aday}}((\source{:}k, \Rday{\aday}^-)) = (\source{:}k, \Rday{\aday})
  \).

  $\bullet$ Second, for an arc \( (\destination, \destination, \Rday{\aday}) \), we have
  to consider two possibilities.  If there exists \( \source\in \SOURCES \) such
  that \( \SFlow{\Rday{\aday}^-}{\source}{\destination} = 1 \), then
  \( \theta_{\Rday{\aday}}((\source, \destination, \Rday{\aday}^-)) = (\destination,
  \destination, \Rday{\aday}) \) using Equation~\eqref{eq:thetaRday2}.
  Otherwise, we have that
  \( \SFlow{\Rday{\aday}^-}{\source}{\destination} = 0 \) for all
  \( \source \in \SOURCES \) and it follows using
  Equation~\eqref{eq:storage_in_storage_out_destination} that
  \( \SFlow{\Rday{\aday}}{\destination}{\source} = 0 \) for all
  \( \source \in \SOURCES \) and therefore that
  \( \theta_{\Rday{\aday}}((\destination, \destination, \Rday{\aday}^-)) =
  (\destination, \destination, \Rday{\aday}) \) using
  Equation~\eqref{eq:thetaRday3}.
  
  $\bullet$ Third, for an arc of the form \( (\destination, \source, \Rday{\aday}) \), we
  have again to consider two possibilities.  If
  $\SFlow{\Rday{\aday}}{\destination}{\source} = 1$, then
  \( \theta_{\Rday{\aday}}((\destination, \destination, \Rday{\aday}^-)) =
  (\destination, \source, \Rday{\aday}) \) using Equation~\eqref{eq:thetaRday3}.
  Otherwise, by definition of $\hbijection{\destination}{\Rday{\aday}}$, there
  exists $\source'$ such that
  $\hbijection{\destination}{\Rday{\aday}}(\source')=\source$, and therefore we
  have using Equation~\eqref{eq:thetaRday2} that
  \( \theta_{\Rday{\aday}}((\source', \destination, \Rday{\aday}^-)) =
  (\destination,\hbijection{\destination}{\Rday{\aday}}(\source'), \Rday{\aday})
  = (\destination, \source, \Rday{\aday}) \).
  
  We conclude that all the arcs in \( \ARCS^{\TEG}_{\Rday{\aday}} \) are in the image of
  \( \theta_{\Rday{\aday}} \) and thus \( \theta_{\Rday{\aday}} \) is surjective.
  
  It remains to prove that $\theta_{\Rday{\aday}}(\arc)$ preserves the value of the storage flow.
  We consider the three possible cases enumerated in Equation~\eqref{eq:thetaRday}.
  \begin{enumerate}
  \item If \( \arc = (\source{:}k, \Rday{\aday}^-) \), we have
    \[
      \Storageflow_\arc = \SFlow{\Rday{\aday}^-}{\source}{k}=
      \SFlow{\Rday{\aday}}{\source}{k}=\Storageflow_{\theta_{\Rday{\aday}}(\arc)}
    \] using Equation~\eqref{eq:no_refill_day} and Equation~\eqref{eq:thetaRday1}.
    
  \item If \( \arc = (\source, \destination, \Rday{\aday}^-) \) we have two subcases to consider.
    If \( \Storageflow_\arc = 1 \) then we have
    \[ \Storageflow_\arc =
      \SFlow{\Rday{\aday}^-}{\source}{\destination}
      = 1 = \SFlow{\Rday{\aday}}{\destination}{\destination}= \Storageflow_{\theta_{\Rday{\aday}}(\arc)}
      \eqsepv 
    \]
    using Equation~\eqref{eq:always_storage_destination} and Equation~\eqref{eq:thetaRday2}.
    Otherwise, we have
    \[
      \Storageflow_\arc=\SFlow{\Rday{\aday}^-}{\source}{\destination}=
      \SFlow{\Rday{\aday}}{\destination}{\hat{\sigma}_{\destination,\Rday{\aday}}(\source)}
      = \Storageflow_{\theta_{\Rday{\aday}}(\arc)}
      \eqfinv
    \]
    using the definition of $\hat{\sigma}_{\destination,\Rday{\aday}}$ and Equation~\eqref{eq:thetaRday2}.
        
  \item If \( \arc = (\destination, \destination, \Rday{\aday}^-) \), we again
    consider two subcases.  If there exists \( \source \) such that
    \( \SFlow{\Rday{\aday}}{\destination}{\source} = 1 \), we obtain
    \[ \Storageflow_\arc 
      = \SFlow{\Rday{\aday}^-}{\destination}{\destination}
      = 1 = \SFlow{\Rday{\aday}}{\destination}{\source}
      = \Storageflow_{\theta_{\Rday{\aday}}(\arc)}
      \eqsepv
    \] using
    Equation~\eqref{eq:always_storage_destination} and Equation~\eqref{eq:thetaRday3}.
    Otherwise, we have that
    \[ \Storageflow_\arc
      = \SFlow{\Rday{\aday}^-}{\destination}{\destination}
      = 1 = \SFlow{\Rday{\aday}}{\destination}{\destination}
      = \Storageflow_{\theta_{\Rday{\aday}}(\arc)}
    \eqfinv
  \]
  using Equation~\eqref{eq:always_storage_destination} and Equation~\eqref{eq:thetaRday3}.
\end{enumerate}
 
Thus, in all cases, \( \theta_{\Rday{\aday}} \) preserves the value of the storage flow which concludes the proof.
\end{proof}
We prove in Lemma~\ref{lemma:preservation_actif_flow} that the number of active
storage flows for each day is constant and equal to the number of storages
$\cardinal{\STORAGES}$. This result serves as a control test, as our flow
formulation is designed to ensure that the number of active storage flows is
globally conserved.

\begin{lemma}\label{lemma:preservation_actif_flow}
  For all $j \in \EDAYS$, we have that
  \begin{equation}
    \sum_{\arc \in\ARCS^{\TEG}_j} \Storageflow_a = \cardinal{\STORAGES}
    \eqfinp
    \label{eq:active_flows_B}
  \end{equation}
\end{lemma}
\begin{proof} We prove Equality~\eqref{eq:active_flows_B} by induction.
  The equality is satisfied by assumption for $\aday= \overline{0}$
  (see Equation~\eqref{eq:initial_position_storage_set}).
  Now assume that it is true for $\aday \in \EDAYS$.
  We have
  \begin{align*}
    \sum_{\arc \in\ARCS^{\TEG}_{j^+}} \Storageflow_a
    =
    \sum_{\arc \in\ARCS^{\TEG}_j} \Storageflow_{\theta_j(a)}
    = \sum_{\arc \in\ARCS^{\TEG}_j} \Storageflow_{a}
    =  \cardinal{\STORAGES}
    \eqfinv
  \end{align*}
  where the first equality is obtained as $\theta_j$ is bijective, the second one
  as $\Storageflow_{\arc}= \Storageflow_{\theta_{\aday}}(\arc)$ and the last one by the induction assumption.
\end{proof}

\begin{definition}\label{definition:transport_planning}
A transport planning $\TransportPlanning$ is an arc path in the time-expanded graph $\GRAPH^{\TEG}$,
that is a sequence of arcs $\TransportPlanning= \nseqp{a_j}{j \in \na{\overline{0},\ldots, \overline{J}}}$ satisfying
\begin{subequations}
  \label{eq:transport_planning}
\begin{align}
  \headarc{a_{j}} &= \tailarc{a_{{j}^+}}\eqsepv \forall j \in \na{\overline{0},\ldots, \underline{J}}
  \\
  \dayarc{a_{j}}&=j \text{ and } \Storageflow_{a_j}= 1\eqsepv \forall j \in \EDAYS
                  \eqfinp
\end{align}
\end{subequations}

For an arc $\arc$, $\tailarc{\arc}$ denotes its tail node, $\headarc{\arc}$ its head node and $\dayarc{\arc}$ the time index of its tail node, as introduced in Definition~\ref{def:arc_properties}.
\end{definition}

The cardinality of the set
$\nset{a\in \ARCS^{\TEG}_{\overline{0}}}{\Storageflow_a =1}$ being
$\cardinal{\STORAGES}$ using Equation~\eqref{eq:initial_position_storage_set},
we denote its distinct elements as
$\na{a_{\overline{0},1}, \ldots , a_{\overline{0}, \cardinal{\STORAGES}}}$.  Now, we
build $\cardinal{\STORAGES}$ transport planning as described by the following
Proposition~\ref{prop:tp}.

\begin{proposition}
  \label{prop:tp}
  For $\storage \in \na{1,\ldots, \cardinal{\STORAGES}}$, we consider the arc path 
  $\TransportPlanning^\storage= \nseqp{a^\storage_j}{j \in \na{\overline{0},\ldots, \overline{J}}} $,
  defined recursively by  $a^\storage_{\overline{0}} = a_{\overline{0},\storage}$
  and for all $j$ in $(\LDAYS \cup \RDAYS)$ by
  \begin{equation*}
    a^\storage_{j}= \theta_{j}\np{a^\storage_{j^{-}}} \eqfinv
  \end{equation*}
  where $\theta_{\aday}$ is defined by~\eqref{eq:thetaLday} when $j\in\LDAYS$ and
  by~\eqref{eq:thetaRday} when $j\in\RDAYS$.
  We have the following properties.
  \begin{enumerate}[label=$B_{\arabic*}$]
  \item\label{P1} For all $\storage \in \na{1,\ldots, \cardinal{\STORAGES}}$, the arc path
    $\TransportPlanning^\storage$ is a transportation planning as
    defined in Definition~\ref{definition:transport_planning};
  \item\label{P2} For two distinct values $\storage,\storage' \in \STORAGES$,
    the paths $\TransportPlanning^\storage$ and $\TransportPlanning^{\storage'}$ do not have any arc in common;
  \item\label{P3} For each active arc $\arc~\in~\ARCS^{\TEG}$ such that $\Storageflow_\arc=1$, there
    exists a unique storage $\storage \in \STORAGES$ such that
  $\arc \in \TransportPlanning^\storage$.
  \end{enumerate}
\end{proposition}

\begin{proof} (\ref{P1}) Fix $\storage \in \STORAGES$. We have to check that Equation~\eqref{eq:transport_planning}
  is satisfied for $\TransportPlanning^\storage$. This follows from Lemmas~\ref{lem:theta_under}
  and~\ref{lem:theta_above} together with the fact that at initial time
  $\Storageflow_{a_{\Rday{0}},\storage}=1$ and the flows are preserved by the mapping $\theta_j$.
  
  (\ref{P2}) Fix two distinct values $\storage,\storage' \in \STORAGES$. 
  The transport plannings $\TransportPlanning^\storage$ and $\TransportPlanning^{\storage'}$
  differ at the initial time $\overline{0}$ by construction and then differ at all times as
  all the mappings $\theta_\aday$ for $j \in (\LDAYS \cup \RDAYS)$ are injective.

  (\ref{P3}) Fix an arc $\arc\in\ARCS^{\TEG}$ such that $\Storageflow_\arc=1$ 
  and consider $j = \dayarc{a}$.
  Using (\ref{P1}), we obtain that the arcs at time $j$ belonging to the $\cardinal{\STORAGES}$ transport plannings
  $\na{\TransportPlanning^\storage}_{\storage \in \STORAGES}$
  have a storage flow equal to one and, using~(\ref{P2}), we obtain that they are all distincts.
  Now, by~Lemma~\ref{lemma:preservation_actif_flow}, the cardinality of set of arcs at time $j$
  with a storage flow equal to one is also $\cardinal{\STORAGES}$. Thus, 
  as $\Storageflow_\arc=1$, there exists a unique storage $\storage \in \STORAGES$ such that
  $\arc \in \TransportPlanning^\storage$, which ends the proof.
\end{proof}

To conclude, we have shown how to construct $\cardinal{\STORAGES}$ distinct
transport plannings $\{\TransportPlanning^\storage\}_{\storage \in \STORAGES}$
that include all active storage flows in $\Storageflow$, where $\Storageflow$ is
a feasible storage flow of Problem~\eqref{flow_model_problem_formulation}. Algorithm~\ref{Algorithm_transport_planning} gives the pseudocode to determine $\TransportPlanning^\storage$ for a given storage $\storage\in \STORAGES$. The physical transport planning for a storage $\storage$ can be obtained by taking the head of each arc in $\TransportPlanning^\storage$.

\begin{algorithm*}
  \caption{From flow to storage transport planning}
  \label{Algorithm_transport_planning}
  \begin{algorithmic}[1] \Require Initial position of the storage $a^{\storage}_{\Rday{0}}$

    \State $\TransportPlanning^\storage \gets (\arc^{\storage}_{\Rday{0}})$ 
    
    \If{$\headarc{\arc^{\storage}_{\Rday{0}}} \in \STATIONS$}
      
      \texttt{at\_destination} $\gets$ \texttt{true}
      
    \Else  \Comment{$\headarc{\arc^{\storage}_{\Rday{0}}} \in \SOURCES$}
      
      \texttt{at\_destination} $\gets$ \texttt{false}
    \EndIf

    \For{$\aday = 1,..,J$}

      \If{\texttt{at\_destination}}
        \State $\destination \gets \headarc{a^{\storage}_{\Lday{\aday}^-}}$
        \State Append arc $(\destination,\destination,\Lday{\aday})$ to $\TransportPlanning^\storage$
        
    \If{$\sum\limits_{\source \in \SOURCES} \SFlow{\Lday{\aday}}{\source}{\destination}=0$}
      \Comment{No swap occurs at destination $\destination$}
    
      \State Append arc $(\destination,\destination,\Rday{\aday})$ to $\TransportPlanning^\storage$
      \Comment{The storage stays at destination $d$}
      
    \Else \Comment{A swap occurs at destination $d$}
    
      \State Find (the unique) source $\source'$ such that
      $\SFlow{\Lday{\aday}}{\destination}{\source'}=1$
    
      \State Append arc $(\destination,\source',\Rday{\aday})$ to $\TransportPlanning^\storage$ \Comment{The storage is sent to source $\source'$}
      \State \texttt{at\_destination} $\gets$ \texttt{false}
    
    \EndIf

    \Else \Comment{\texttt{at\_destination}=false}
      \State $\source \gets \headarc{a^{\storage}_{\Lday{\aday}^-}}$
      \State Get next arc $(\source,p,\Lday{\aday})$ using $\bijection{\Lday{\aday}}{\source}$ \Comment{see Equation~\eqref{eq:thetaLday}}
      
    \State Append arc $(\source,p,\Lday{\aday})$ to $\TransportPlanning^\storage$
            
    \If{\texttt{p} $\in \STATIONS$}   \Comment{The storage is sent to destination $d$}   

    \State Append arc $(p,p,\Rday{\aday})$ to $\TransportPlanning^\storage$

    \State \texttt{at\_destination} $\gets$ \texttt{true}

  \Else  \Comment{The storage stays at source $\source$}

    \State Append arc $(\source,p,\Rday{\aday})$ to $\TransportPlanning^\storage$
    
    \EndIf

    \EndIf
    
    \EndFor  
    
    \State \Return  $\TransportPlanning^\storage$
  \end{algorithmic}
\end{algorithm*}

\section{Conclusion}\label{se:conclusion}

In this paper, we tackle an hydrogen transportation problem ---in which mobile
storages are dynamically routed and swapped--- that we called PRP-MI. Modeling the
problem using storage (binary) and hydrogen (continuous) flows on a
time-expanded graph, we obtain a MILP formulation that captures both the daily
scheduling of the mobile storages and the management of inventories. The
\algname{MA} method (direct use of a MILP solver) successfully handles small to medium problem instances within reasonable computational times, and outperforms the greedy heuristic,
which is inspired by the heuristic employed by the industrial partner. For
larger instances, we introduced the \algname{RH} method, that first determines a
transport planning of the mobile storages, and then, finds the quantity of hydrogen to purchase at the
sources and how to allocate it to fill the storages. Numerical experiments
indicate that this approach achieves two main results: it provides better
results compared to \algname{MA} for large instance, and it obtains tight lower
bounds, while remaining computationally tractable, demonstrating its practical
suitability for hydrogen logistics. Future research should integrate
stochasticity, as for example in hydrogen demand and in mobile storage transportation times, using similar technics as the ones used in~\cite{StochasticAxel}.

\setlength{\parskip}{6pt plus 2pt}
\appendix
\section{Linearization of constraints}\label{appendix:linearization}
\label{hympulsion_linearization} In this Appendix, we prove
Lemma~\ref{lemma:linearization}, by giving a linear reformulation of
Constraints~\eqref{eq:flow_arc}, \eqref{eq:emptying_bounded1},
\eqref{eq:swap_destination}, \eqref{eq:demand_satisfaction2},
\eqref{eq:swap_time}, \eqref{eq:emtying_bounded2},
\eqref{eq:permutation_constraints}, and \eqref{eq:forced_arc_source} by using
additional variables and constraints.

$\bullet$ First, we consider the implications constraints~\eqref{eq:flow_arc}
and~\eqref{eq:swap_time} and~\eqref{eq:forced_arc_source} and just show how to linearize Constraint~\eqref{eq:flow_arc} as the
same mechanism applies to Constraint~\eqref{eq:swap_time} and~\eqref{eq:forced_arc_source}.
The implication constraint
\begin{equation*}
  \Hydrogenflow_a> 0 \implies \Storageflow_a = 1
  \eqsepv \forall a \in \ARCS^{\TEG} \eqfinp
\end{equation*}
is equivalent to the following linear constraint
\begin{equation*}
  \Hydrogenflow_a
  \leq \StockMax\Storageflow_a
  \eqsepv \forall a \in \ARCS^{\TEG} \eqfinv
\end{equation*}
where $\StockMax$ is the hydrogen flow upper bound (defined in~\ref{Adeux}).  Indeed, if
$\Hydrogenflow_a > 0$, this inequality forces $\Storageflow_a = 1$, because
otherwise the right-hand side would be zero, contradicting the
inequality. Conversely, if $\Hydrogenflow_a = 0$, the variable $\Storageflow_a$
is not constrained and can take any value in $\{0,1\}$.

$\bullet$ Second, we address the linearization of the product between a binary
variable and a continuous variable present in the right hand side of both
Constraints~\eqref{eq:swap_destination} and~\eqref{eq:demand_satisfaction2}.  To
linearize the product, we introduce a new variable denoted by $p$ together with
an equality constraint
\begin{align}
  p = 
  \overbrace{\findi{\na{0}}\bp{\sum\limits_{\source \in \SOURCES}
  \SFlow{\Lday{\aday}}{\source}{\destination}}}^{b}
  \HFlow{\Lday{\aday}}{\destination}{\destination}
  \eqfinv
  \label{eq:product_to_linearize}
\end{align}
which remains to be linearized. Noting that
the continuous variable $\HFlow{\Lday{\aday}}{\destination}{\destination}$
is bounded as it belongs to $[0,\StockMax]$, the added
nonlinear equality constraint~\eqref{eq:product_to_linearize}
is equivalent to the following set of linear inequalities.
\begin{subequations}
  \label{eq_prod_linearized}
  \begin{align}
    &p \leq\HFlow{\Lday{\aday}}{\destination}{\destination} \eqfinv
      \label{eq:linear_prod_continuous1}
  \\
  &p
    \leq\StockMax b \eqfinv
      \label{eq:linear_prod_continuous2}
  \\
  &0 \leq p \eqfinv
      \label{eq:linear_prod_continuous3}
  \\
    &\HFlow{\Lday{\aday}}{\destination}{\destination} -\StockMax (1- b)
      \leq p \eqfinp
      \label{eq:linear_prod_continuous4}  
\end{align}
\end{subequations}
Indeed, we consider two cases. If $b=0$,
Constraint~\eqref{eq:linear_prod_continuous2} becomes $p \leq 0$, which combined
with Constraint~\eqref{eq:linear_prod_continuous3}, implies that $p = 0$.  Then,
Constraint~\eqref{eq:linear_prod_continuous4} reduces to
$\HFlow{\Lday{\aday}}{\destination}{\destination} \le \StockMax$ and is satisfied.
Otherwise, if $b=1$, Constraint~\eqref{eq:linear_prod_continuous4} simplifies to
$\HFlow{\Lday{\aday}}{\destination}{\destination} \leq p$ which combined with
Constraint~\eqref{eq:linear_prod_continuous1}, gives
$p =\HFlow{\Lday{\aday}}{\destination}{\destination}$. Then,
Constraint~\eqref{eq:linear_prod_continuous2} boils down to
$\HFlow{\Lday{\aday}}{\destination}{\destination} \le \StockMax$ which is
satisfied. We conclude that
Constraints~\eqref{eq_prod_linearized} and~\eqref{eq:product_to_linearize} are
equivalent as they share the same solutions for their common variables.

$\bullet$  Third, we move to Constraints~\eqref{eq:emptying_bounded1} and
\eqref{eq:emtying_bounded2} and just show how to linearize
Constraint~\eqref{eq:emptying_bounded1} as the same mechanism applies to
Constraint~\eqref{eq:emtying_bounded2}. For $\aday \in \DAYS$ and
$\destination \in \STATIONS$, the Constraint~\eqref{eq:emptying_bounded1}
reads as follows
\begin{equation}
  \nodeV{z}{\Lday{\aday}}{\destination}
  =\min( \HFlow{\Lday{\aday}^-}{\destination}{\destination},
  \nodeV{\Demand}{\Lday{\aday}}{\destination}) \eqfinp
  \label{eq:emptying_bounded1_copy}
\end{equation}
By adding a binary variable $b$ and noting that the continuous variables
$\HFlow{\Lday{\aday}}{\destination}{\destination}$ and
$\nodeV{\Demand}{\Lday{\aday}}{\destination}$ both belong to $[0,\StockMax]$
(using Assumption~\ref{Adeux} for
$\nodeV{\Demand}{\Lday{\aday}}{\destination}$),
we show that Constraint~\eqref{eq:emptying_bounded1_copy} is equivalent
to the following set of linear constraints~\eqref{eq:min_non_linear_c}
\begin{subequations}
  \label{eq:min_non_linear_c}
  \begin{align}
  &  \nodeV{z}{\Lday{\aday}}{\destination} \leq  \HFlow{\Lday{\aday}^-}{\destination}{\destination} \eqfinv
      \label{eq:linear_min1}
  \\
  & \nodeV{z}{\Lday{\aday}}{\destination}
    \leq  \nodeV{\Demand}{\Lday{\aday}}{\destination} \eqfinv
      \label{eq:linear_min2}
  \\
  & \nodeV{z}{\Lday{\aday}}{\destination} \geq \HFlow{\Lday{\aday}^-}{\destination}{\destination} - b\StockMax \eqfinv
      \label{eq:linear_min3}
  \\
  &\nodeV{z}{\Lday{\aday}}{\destination} \geq  \nodeV{\Demand}{\Lday{\aday}}{\destination} - (1-b)\StockMax \eqfinp
      \label{eq:linear_min4}  
\end{align}
\end{subequations}
First, assume that Constraint~\eqref{eq:emptying_bounded1_copy} is satisfied.
Then, we have two possible cases or $(i)$
$\nodeV{z}{\Lday{\aday}}{\destination}=\nodeV{\Demand}{\Lday{\aday}}{\destination}
\le \HFlow{\Lday{\aday}^-}{\destination}{\destination}$ or $(ii)$
$\nodeV{z}{\Lday{\aday}}{\destination}=
\HFlow{\Lday{\aday}^-}{\destination}{\destination} \le
\nodeV{\Demand}{\Lday{\aday}}{\destination}$. Choosing $b=1$ in case $(i)$ and
$b=0$ in case $(ii)$ straightfowardly give a solution to
Constraints~\eqref{eq:min_non_linear_c}. Second,
the solutions of Constraint~\eqref{eq:min_non_linear_c} are the
following. When $b=0$, Constraint~\eqref{eq:linear_min4} is always satistied and
from Constraints~\eqref{eq:linear_min1}, \eqref{eq:linear_min2} and~\eqref{eq:linear_min3}
we obtain that $\nodeV{z}{\Lday{\aday}}{\destination}=\nodeV{\Demand}{\Lday{\aday}}{\destination}
\le \HFlow{\Lday{\aday}^-}{\destination}{\destination}$.
When $b=1$, Constraint~\eqref{eq:linear_min3} is always satistied and
from Constraints~\eqref{eq:linear_min1}, \eqref{eq:linear_min2} and~\eqref{eq:linear_min4}
we obtain that
$\nodeV{z}{\Lday{\aday}}{\destination}= \HFlow{\Lday{\aday}^-}{\destination}{\destination}
\le \nodeV{\Demand}{\Lday{\aday}}{\destination}$. Thus, 
in both cases we obtain that Constraint~\eqref{eq:emptying_bounded1_copy} is satisfied.

$\bullet$ 
Finally, we fix $\aday \in \DAYS$ and
$\source \in \SOURCES$ and address the linearization of the nondecreasing
stock constraint~\eqref{eq:permutation_constraints} which is
satisfied if there exists a bijection
$\bijection{\Lday{\aday}}{\source}: \PermutationBijectionDomain_{\source,\lday} \to
\PermutationBijectionCoDomain_{\source,\lday}$ such that
\begin{equation}
  \label{eq:permutation_constraints_bis}
  \Flow{f^{\mathsf{in}}}{\Lday{\aday}}{\source}{l} \leq
  \Flow{f^{\mathsf{out}}}{\Lday{\aday}}{\source}{\bijection{\Lday{\aday}}{\source}(l)}
 \eqsepv \forall l \in \PermutationBijectionDomain_{\source,\lday} \eqfinp
\end{equation}

We introduce a matrix $\np{\beta_{n,m}}_{n,m \in \MaxArcsSource}$ of binary variables,
where $\MaxArcsSource$ is defined in Equation~\eqref{eq:MaxArcsSource}, and we
prove that Constraint~\eqref{eq:permutation_constraints_bis} is equivalent to
the set of constraints~\eqref{sigma_equiv} defined by
\begin{subequations}
  \label{sigma_equiv}
  \begin{align}
    \sum\limits_{m \in \MaxArcsSource} \beta_{n,m}
    &= \Flow{y^{\mathsf{in}}}{\Lday{\aday}}{\source}{n}  \eqsepv \forall n \in \MaxArcsSource \eqfinv
      \label{permutation_constraint_linear_1}
    \\
    \sum\limits_{n \in \MaxArcsSource} \beta_{n,m}
    &=  \Flow{y^{\mathsf{out}}}{\Lday{\aday}}{\source}{m} \eqsepv \forall m \in  \MaxArcsSource \eqfinv
      \label{permutation_constraint_linear_2}
    \\
    \Flow{f^{\mathsf{in}}}{\Lday{j}}{\source}{n}
    \leq \Flow{f^{\mathsf{out}}}{\Lday{j}}{\source}{m}
    &+ (1-\beta_{n,m})\StockMax \eqsepv \forall n,m \in \MaxArcsSource  \eqfinp
      \label{permutation_constraint_linear_3}
  \end{align}
\end{subequations}
\balance
We start be proving that when Constraint~\eqref{sigma_equiv} is satisfied, there
exists a solution to Constraint~\eqref{eq:permutation_constraints_bis}.  For
that purpose, consider a matrix $\beta$ and four vectors $f^{\mathsf{in}}$,
$f^{\mathsf{out}}$, $y^{\mathsf{in}}$ and $y^{\mathsf{out}}$ satisfying
Constraints~\eqref{sigma_equiv}. For all
$n \not\in \PermutationBijectionDomain_{\source,\lday}$ we have by definition of
$\PermutationBijectionDomain_{\source,\lday}$ that
$\Flow{y^{\mathsf{in}}}{\Lday{\aday}}{\source}{n}=0$ and therefore using
Constraint~\eqref{permutation_constraint_linear_1} that $\beta_{n,m}=0$ for all
$m\in\MaxArcsSource$. In a similar way, for all
$m\not\in \PermutationBijectionCoDomain_{\source,\lday}$ we have
$\Flow{y^{\mathsf{out}}}{\Lday{\aday}}{\source}{m}=0$ and therefore using
Constraint~\eqref{permutation_constraint_linear_2} that $\beta_{n,m}=0$ for all
$n\in\MaxArcsSource$. We therefore obtain that, for all
$n \in\PermutationBijectionDomain_{\source,\lday}$,
\begin{align*}
  \sum\limits_{m \in \PermutationBijectionCoDomain_{\source,\lday}} \beta_{n,m}
  +
  \underbrace{\sum\limits_{m \not\in \PermutationBijectionCoDomain_{\source,\lday}} \beta_{n,m}}_{=0}
  &=  \Flow{y^{\mathsf{in}}}{\Lday{\aday}}{\source}{n} =  1 \eqsepv
\end{align*}
and, for all $m \in \PermutationBijectionCoDomain_{\source,\lday}$,
\begin{align*}
  \sum\limits_{n \in \PermutationBijectionDomain_{\source,\lday}} \beta_{n,m}
  +
  \underbrace{\sum\limits_{n \not\in \PermutationBijectionDomain_{\source,\lday}} \beta_{n,m}}_{=0}
  &=\Flow{y^{\mathsf{out}}}{\Lday{\aday}}{\source}{m} = 1 \eqfinp
\end{align*}
Thus, the matrix $\beta$ restricted to
$\PermutationBijectionDomain_{\source,\lday}{\times}%
\PermutationBijectionCoDomain_{\source,\lday}$ is a permutation matrix as a
square binary matrix that has exactly one entry of 1 in each row and each column
with all other entries~0. All the elements of the matrix $\beta$ outside its
restriction to $\PermutationBijectionDomain_{\source,\lday}{\times}%
\PermutationBijectionCoDomain_{\source,\lday}$ have value zero. Now, we build a
bijection
$\bijection{\Lday{\aday}}{\source}: \PermutationBijectionDomain_{\source,\lday}
\to \PermutationBijectionCoDomain_{\source,\lday}$ as follows
\begin{equation}
  \forall n \in \PermutationBijectionDomain_{\source,\lday}
  \eqsepv
  \na{\bijection{\Lday{\aday}}{\source}(n)} =
  \argmax_{m \in \PermutationBijectionCoDomain_{\source,\lday}} \beta_{n,m}
  \eqfinp
\end{equation}
We consider Constraint~\eqref{permutation_constraint_linear_3}.
First, when $\beta_{n,m}=0$, this constraint reduces to
\begin{align}
  \Flow{f^{\mathsf{in}}}{\Lday{j}}{\source}{n}
  \leq \Flow{f^{\mathsf{out}}}{\Lday{j}}{\source}{m}
  &+ \StockMax \eqsepv \forall m \in \MaxArcsSource  \eqfinp
    \label{eq_ok}
\end{align}
which is satisfied as both $\Flow{f^{\mathsf{in}}}{\Lday{j}}{\source}{n}$ and
$\Flow{f^{\mathsf{out}}}{\Lday{j}}{\source}{m}$ are in $[0,\StockMax]$.  Second,
when considering the subsets of indices for which $\beta_{n,m}=1$,
Constraint~\eqref{permutation_constraint_linear_3} boils down to
\begin{align}
  \Flow{f^{\mathsf{in}}}{\Lday{j}}{\source}{n}
  \leq \Flow{f^{\mathsf{out}}}{\Lday{j}}{\source}{\bijection{\Lday{\aday}}{\source}(n)} 
  \eqsepv \forall n \in \PermutationBijectionDomain_{\source,\lday}  \eqfinv
\end{align}
which is precisely Constraint~\eqref{eq:permutation_constraints_bis}.

We prove the converse, that is when
Constraints~\eqref{eq:permutation_constraints_bis} is satisfied, there exists a
matrix $\beta$ satisfying Constraint~\eqref{sigma_equiv}.  We consider a bijection
$\bijection{\Lday{\aday}}{\source}: \PermutationBijectionDomain_{\source,\lday}
\to \PermutationBijectionCoDomain_{\source,\lday}$ satisfying
Constraint~\eqref{eq:permutation_constraints_bis} and build a matrix $\beta$ as
follows
\begin{equation}
  \beta_{n,m}=
  \begin{cases}
    \findi{\na{m}}(\bijection{\Lday{\aday}}{\source}(n))
    &: n \in\PermutationBijectionDomain_{\source,\lday}
    \eqfinv\\
    0
    &: n \not\in\PermutationBijectionDomain_{\source,\lday}
      \eqfinp
  \end{cases}
\end{equation}
Using the definition of $\beta$ and the fact that
$\bijection{\Lday{\aday}}{\source}$ is a bijection it is straightforward to
obtain that $\beta$, $f^{\mathsf{in}}$ and $f^{\mathsf{out}}$ satisfy
Constraint~\eqref{sigma_equiv}.

\end{document}